\documentclass{article}

\usepackage{authblk}

\usepackage[a4paper, total={7in, 8.5in}]{geometry}

\usepackage{amssymb,amsmath,amsfonts,amsthm}
\usepackage{color}
\usepackage[dvipsnames]{xcolor}
\usepackage{graphicx}
\usepackage{placeins}
\usepackage[dvipsnames]{xcolor}

\newtheorem{theorem}{Theorem}[section]
\newtheorem{lemma}[theorem]{Lemma}
\newtheorem{proposition}[theorem]{Proposition}

\newtheorem{remark}[theorem]{Remark} 
\newtheorem{example}[theorem]{Example}

\newtheorem*{conjecture}{Conjecture}
\newtheorem{definition}[theorem]{Definition}

\numberwithin{equation}{section}

\renewcommand{\epsilon}{\varepsilon}
\DeclareMathOperator{\dvg}{div} 
\DeclareMathOperator{\grad}{grad} 

\newcommand{\RR}{\mathbb{R}}
\newcommand*{\CC}{\mathbb{C}}

\newcommand*{\NN}{\mathbb{N}}

\usepackage{mathtools}
\usepackage{hyperref}
\hypersetup{hidelinks}
\newcommand{\lie}{\mathcal{L}}
\newcommand{\dvm}{\,{\rm dvol}_M}
\newcommand{\dvpm}{\,{\rm dvol}_{\partial M}}
\DeclareMathOperator{\hess}{{\rm Hess}}
\newcommand{\half}{{\frac{1}{2}}}
\newcommand{\pam}{{\partial M}}

\title{A note on the magnetic Steklov operator on functions}

\author[1]{Tirumala Chakradhar\thanks{\texttt{tirumala.chakradhar@bristol.ac.uk}}}
\author[2]{Katie Gittins\thanks{\texttt{katie.gittins@durham.ac.uk}}}
\author[3,4]{Georges Habib\thanks{\texttt{ghabib@ul.edu.lb}}}
\author[2]{Norbert Peyerimhoff\thanks{\texttt{norbert.peyerimhoff@durham.ac.uk}}}

\affil[1]{\footnotesize School of Mathematics, University of Bristol, Fry Building, Woodland Road, BS8 1UG, UK}
\affil[2]{\footnotesize Department of Mathematical Sciences, Durham University, Mathematical Sciences and Computer Science Building, Upper Mountjoy Campus, Stockton Road, Durham University, DH1 3LE, United Kingdom}
\affil[3]{\footnotesize Lebanese University, Faculty of Sciences II, Department of Mathematics, P.O. Box 90656 Fanar-Matn, Lebanon}
\affil[4]{\footnotesize Universit\'e de Lorraine, CNRS, IECL, 54506 Nancy, France}

\date{}

\begin{document}

\maketitle

\vspace*{-1cm}

\begin{abstract}
We consider the magnetic Steklov eigenvalue problem on compact Riemannian manifolds with boundary for generic magnetic potentials and establish various results concerning the spectrum. We provide equivalent characterizations of magnetic Steklov operators which are unitarily equivalent to the classical Steklov operator and study bounds for the smallest eigenvalue. We prove a Cheeger-Jammes type lower bound for the first eigenvalue by introducing magnetic Cheeger constants. We also obtain an analogue of an upper bound for the first magnetic Neumann eigenvalue due to Colbois, El Soufi, Ilias and Savo. In addition, we compute the full spectrum in the case of the Euclidean $2$-ball and $4$-ball for a particular choice of magnetic potential given by Killing vector fields, and discuss the behavior. Finally, we establish a comparison result for the magnetic Steklov operator associated with the manifold and the square root of the magnetic Laplacian on the boundary, which generalizes the uniform geometric upper bounds for the difference of the corresponding eigenvalues in the non-magnetic case due to Colbois, Girouard and Hassannezhad.
\end{abstract}

\tableofcontents

\section{Introduction}

In this paper, we consider spectral properties of the magnetic Steklov operator on smooth complex functions. 
The magnetic Steklov problem on functions has been investigated in various works, see, e.g., \cite{CS:24}, \cite{CPS:22}, \cite{PS:23}, \cite{LT:23}, \cite{EO:22}, \cite{HN:24a}, \cite{HN:24b}, \cite{NSU:95}, \cite{DKSU:07}, \cite{H:18}.
Let $(M^m,g)$ be an $m$-dimensional compact Riemannian manifold with smooth boundary $\partial M$ and let $\eta$ be a $1$-form on $M$, the \emph{magnetic potential}.
The magnetic Laplacian is defined as the composition of the magnetic differential $d^\eta f := d^M f+ i f \eta$ for functions $f \in C^\infty(M)$ and the magnetic co-differential $\delta^\eta \omega := \delta^M \omega - i \eta \lrcorner \omega$ for $1$-forms $\omega \in \Omega^1(M)$, that is,
$\Delta^\eta: C^\infty(M) \to C^\infty(M)$ with
$$ \Delta^\eta f := \delta^\eta (d^\eta f) = \Delta^M f - 2 i \eta(f) + (i \delta^M \eta + |\eta|^2) f, $$
where $\Delta^M = \delta^M d^M$ is the classical (non-magnetic) Laplacian on $M$.
For any $f \in C^\infty(\partial M, \mathbb{C})$,  there exists a unique $\eta$-harmonic extension $\hat f \in C^\infty(M,\mathbb{C})$, that is, 
 \begin{equation*}
		 \left\{
		\begin{matrix} 
			 \Delta^\eta \hat f & = 0,\\
		\hat f\vert_{\partial M} & = f.
		\end{matrix}
		\right.
	\end{equation*}
Note that uniqueness of this extension is a direct consequence of the following maximum principle for magnetic Laplacians, which, in turn, is a consequence of a more general version about vector bundle Laplacians (see Proposition \ref{prop:maxprinc} below): 

\begin{proposition}[magnetic Maximum Principle]
\label{prop:maxprinc}
Let $(M,g)$ be a compact Riemannian manifold with smooth boundary $\partial M$, $\eta \in \Omega^1(M)$ and $f \in C^\infty(M)$ satisfying $\Delta^\eta f = 0$. Then the function
$|f| \ge 0$ assumes its maximum at a boundary point in $\partial M$.
\end{proposition}

The magnetic Steklov operator is then defined for any  $f \in C^\infty(\partial M, \mathbb{C})$ as
$$ T^\eta(f) =-\nu\lrcorner d^\eta \hat f, $$
where $\nu \in C^\infty(\partial M,TM)$ is the inward unit-normal vector field of the boundary.
The magnetic Steklov operator has a discrete spectrum (see, e.g, \cite[Appendix A.3]{CPS:22}) denoted by 
$$ 0 \le  \sigma_1^\eta(M) \le \sigma_2^\eta(M) \le \cdots. $$
Moreover, the lowest eigenvalue $\sigma_1^\eta(M)$ has the following variational characterization (see, e.g., \cite[(22)]{CPS:22})
\begin{equation}\label{eq:characsigma}
 \sigma_1^\eta(M) = \inf_{f \in C^\infty(M,\mathbb{C})} \frac{\int_M |d^\eta  f|^2 {\rm dvol}_M}{\int_{\partial M} |f|^2 {\rm dvol}_{\partial M} }.
 \end{equation}
In this paper, we present results about the eigenvalues 
of the Steklov operator. Before presenting our results, we would like to mention a remarkable similarity between the magnetic Steklov eigenvalue problem, which can be concisely formulated by
$$ \left\{
		\begin{matrix} 
			 \Delta^\eta F & = 0,\\
		\nu \lrcorner d^\eta F &= \sigma F,
		\end{matrix}
		\right.$$
and the magnetic Neumann eigenvalue problem, which has the form
$$ \left\{
		\begin{matrix} 
			 \Delta^\eta F & = \lambda F,\\
		\nu \lrcorner d^\eta F &= 0,
		\end{matrix}
		\right.$$
with its associated discrete spectrum
$$ 0 \le \lambda_{1,N}^\eta(M) \le \lambda_{2,N}^\eta(M) \le \cdots. $$
It turns out that many eigenvalue results in one of these two problems have counterparts in the other problem by the use of very similar arguments. It should be mentioned, however, that the variational characterization of the lowest Neumann eigenvalue is given by
\begin{equation}\label{eq:characlambda}
 \lambda_{1,N}^\eta(M) = \inf_{f \in C^\infty(M,\mathbb{C})} \frac{\int_M |d^\eta  f|^2 {\rm dvol}_M}{\int_M |f|^2 {\rm dvol}_M }.
 \end{equation}
The third eigenvalue problem of relevance is the magnetic eigenvalue problem on the closed manifold $\partial M$ with induced Riemannian metric. Using the inclusion map $\iota: \partial M \to M$ and the $1$-form $\eta_0 = \iota^* \eta \in \Omega^1(\partial M)$, it is described by
$$ \Delta^{\eta_0} f = \lambda_k^{\eta_0}(\partial M) f  $$
with its associated discrete spectrum
$$ 0 \le \lambda_1^{\eta_0}(\partial M) \le \lambda_2^{\eta_0}(M) \le \cdots. $$

We end this introduction by giving a brief overview about this paper. The results in this paper are presented in Section \ref{sec:results}. There we give a characterization of magnetic potentials which can be gauged away in different ways, obtain lower and upper bounds for the smallest magnetic Steklov eigenvalue, compute the full magnetic Steklov spectra for particular magnetic potentials on the Euclidean $2$-ball and $4$-ball, and finally relate the full magnetic Steklov spectrum to the full magnetic Laplace spectrum on the boundary. These results are then successively proved in Sections \ref{sec:maxprincandshike} to \ref{sec:speccomparison}. Moreover, we compute magnetic frustration constants for two $2$-dimensional examples in Subsection \ref{subsec:frustind}. Magnetic frustration constants appear in magnetic versions of Cheeger constants, as explained in Subsection \ref{subsec:lowuppsmalleststeklov}.

\section{Results}
\label{sec:results}

In this section, we present our main results. For convenience, we discuss them in different subsections.

\subsection{Characterizations of magnetic potentials which can be gauged away}

Our first observation is a magnetic Steklov operator analogue of a 
fundamental result, which goes back independently to Shikegawa \cite[Prop. 3.1 and Th. 4.2]{Shi:87} in the case of compact manifolds and Helffer \cite{He:88a,He:88b} (see also references therein) in the case of domains in Euclidean space. To state it, we need to introduce the following space of $1$-forms for a manifold $M$:
$$ \mathfrak{B}_M := \left\{ \alpha_\tau := \frac{d^M\tau}{i \tau}: \tau \in C^\infty(M,\mathbb{S}^1) \right\}, $$
where $\mathbb{S}^1 := \{ z \in \mathbb{C}: |z| = 1 \}$.

\begin{theorem}[Equivalent characterizations of functions in $\mathfrak{B}_M$] \label{thm:shikegawa}
Let $(M^m,g)$ be a compact Riemannian manifold with smooth boundary $\partial M$ and let $\eta$ be a differential $1$-form. We have the equivalences: 
\begin{itemize}
    \item[(i)] $\sigma_1^\eta(M) = 0$;
    \item[(ii)] $\eta \in \mathfrak{B}_M$;
    \item[(iii)] $d^M\eta = 0$ and $\int_C \eta \in 2 \pi \mathbb{Z}$ for any closed curve $C$ in $M$. 
\end{itemize}
Moreover, if one of these equivalent conditions is satisfied, then $T^\eta$ and the classical Steklov operator $T=T^0$ without magnetic potential have the same spectrum.
\end{theorem}

If a magnetic potential $\eta$ lies in $\mathfrak{B}_M$, we say, it can be \emph{gauged away}, since the corresponding magnetic operator is unitarily equivalent to the non-magnetic one and has the same spectrum. We also say that the magnetic potentials $\eta$ and $\eta'$ are \emph{gauge-equivalent}, if $\eta - \eta' \in \mathfrak{B}_M$.

\begin{remark} \label{rem:ShikegawaNeumann}
By the relationship between the magnetic Steklov and Neumann eigenvalue problems mentioned earlier, an analogous proof to that of Theorem \ref{thm:shikegawa} provides the same equivalences for the magnetic Neumann problem where (i) needs to be replaced by $\lambda_{1,N}^\eta(M)=0$.
\end{remark}

\begin{remark} \label{rem:etaeta0relation}
Notice that when $\eta\in \mathfrak{B}_M$, then the restriction $\eta_0 = \iota^*\eta$ is in $\mathfrak{B}_{\partial M}$. Indeed, assume $\eta = \frac{d^M\tau}{i\tau} \in \mathfrak{B}_M$. Then, as $\iota^* d^M=d^{\partial M}\iota^*$, we get $\iota^*\eta=\frac{d^{\partial M}(\tau|_{\partial M})}{i(\tau|_{\partial M})}$, where $\tau|_{\partial M}:\partial M\to \mathbb{S}^1$ is just the restriction of $\tau$. Hence $\iota^*\eta\in \mathfrak{B}_{\partial M}$.

The converse is not true. Given $\eta_0 \in \mathfrak{B}_{\partial M}$, there is not always an $\eta \in \mathfrak{B}_{M}$ such that $\eta_0 = \iota^* \eta$. A counterexample is given by the unit ball $M = \mathbb{B}^2 \subset \mathbb{R}^2$ centered at the origin, $\partial M = \mathbb{S}^1$, and 
$$ \eta_0 = d\theta = \frac{d^{\partial M}\tau_0}{i\tau_0} \in \mathfrak{B}_{\mathbb{S}^1} $$ 
with $\theta: \mathbb{S}^1 \to (\mathbb{R} \mod 2\pi)$ the angle map
and $\tau_0 = e^{i \theta}: \mathbb{S}^1 \to \mathbb{S}^1$ the identity. Any $\eta \in \mathfrak{B}_M$ must be of the form $\frac{d^M \tau}{i \tau}$ with a smooth function $\tau: M \to \mathbb{S}^1$. The condition $\iota^* \eta = \eta_0 = d \theta$ means that the restriction of $\tau$ to $\mathbb{S}^1$ must be (up to a fixed rotation $e^{ic} \in \mathbb{S}^1$) the identity. Then the continuous map $F: [0,1] \times \mathbb{S}^1 \to \mathbb{S}^1$, $F(t,z) = e^{-ic} \tau(tz)$ would be a contraction of $\mathbb{S}^1$ to a point, which is a contradiction. 
\end{remark}

\subsection{Lower and upper bounds for the smallest magnetic Steklov eigenvalue}
\label{subsec:lowuppsmalleststeklov}

We first discuss a Cheeger type inequality for the magnetic Steklov operator, which is a magnetic analogue of a result by Jammes \cite{Ja:15}. The classical Cheeger inequality in \cite{Ch:70} states that the second smallest eigenvalue $\lambda_2(M)$ of the Laplacian $\Delta^M$ of a closed manifold $(M^m,g)$ is bounded below by
$$ \lambda_2(M) \ge \frac{h(M)^2}{4}, $$
where 
\begin{equation} \label{eq:Cheegerclassical} 
h(M) = \inf_{|D| \le \frac{|M|}{2}} \frac{|\partial D|}{|D|}, \end{equation}
and the infimum runs over all open subsets $D \subset M$ with compact closure whose boundary $\partial D$ is a
smooth $(m-1)$-dimensional submanifold which divides $M$ into at least two submanifolds with boundary. This result generalizes to an analogous result for the second smallest Neumann eigenvalue $\lambda_{2,N}(M)$ in the case of manifolds with boundary. Moreover, the restriction $|D| \le \frac{|M|}{2}$ is necessary since, otherwise, choices $D = M \setminus B$ with arbitrary small metric balls $B \subset M$ would lead to vanishing of the Cheeger constant. Cheeger's original inequality is for the second smallest eigenvalue, since the smallest eigenvalue is always zero. In the magnetic case, the smallest eigenvalue is no longer zero if the magnetic potential $\eta$ cannot be gauged away, by Remark \ref{rem:ShikegawaNeumann}, and Cheeger type inequalities are then statements about the smallest eigenvalue. 
However, the isoperimetric ratios on the right-hand side of \eqref{eq:Cheegerclassical} require an additional term involving the magnetic potential. This extra term was introduced in \cite[Def. 7.2]{LLPP:15} (see also \cite{ELMP:16}) under the name ``frustration index". However, frustration indices (which appeared also under the name of ``line index of balance" in \cite{Ha:59}) were originally only defined in the discrete setting of graphs with signed Laplacians (see, e.g., \cite{Za:82} and 
\cite{ZR:10} and references therein) and are used in the other literature only in this very specific context. (There is also a close relation to the notion of ``geometrical frustration" in condensed matter physics.) In the setting of graphs, signed Laplacians are special cases of magnetic Laplacians which, in turn, are special cases of so-called ``connection Laplacians". In the discrete setting, Bandeira et al derived in \cite{BSS:13} various Cheeger-type inequalities for connection Laplacians involving different kinds of ``frustration constants" measuring quantitatively how much a connection differs from the ones which can be gauged away. In view of this background, we think it is best in the smooth setting of magnetic Laplacians on Riemannian manifolds to refer to the additional term $\iota^\eta(D)$ appearing in Definition \ref{def:frustindex} below, generally, as a ``magnetic frustration constant" instead of the specific notion of ``frustration index".


\begin{definition} \label{def:frustindex}
  Let $(M^m,g)$ be a manifold with smooth boundary $\partial M$ and let $\eta \in \Omega^1(M)$. The magnetic frustration constant $\iota^\eta(D)$ of a domain $D \subset M$ is defined as follows:
  $$ \iota^\eta(D) = \inf_{\tau \in C^\infty(D,\mathbb{S}^1)} \int_D | d^\eta\tau| {\rm dvol}_M = \inf_{\alpha \in \mathfrak{B}_D} \int_D |\eta + \alpha| {\rm dvol}_M. $$
  The magnetic Cheeger constants $h^\eta(M)$ and $(h^\eta)'(M)$ are defined as follows:
  $$ h^\eta(M) = \inf_{D} \frac{\iota^\eta(D) + |\partial_I D|}{|D|} \quad \text{and}
 \quad (h^\eta)'(M) = \inf_{D} \frac{\iota^\eta(D)+ |\partial_I D|}{|\partial_E D|}, $$
 where the infimum is taken over all non-empty open subsets $D \subset M$ with compact closure whose boundary $\partial D$ is a smooth $(m-1)$-dimensional submanifold, and where $\partial_I D:= \partial D \cap {\rm{int}}(M)$ and $\partial_E D:= \partial D \cap \partial M$. 
\end{definition}

With these notions, we have the following.

\begin{theorem}[Cheeger type inequality] \label{thm:magnjammes}
   Let $(M^m,g)$ be a compact manifold with smooth boundary $\partial M$ and let $\eta \in \Omega^1(M)$. The smallest eigenvalue of the magnetic Steklov operator $\sigma_1^\eta(M)$ satisfies
   \begin{equation} \label{eq:cheegerjammesineq} 
   \sigma_1^\eta(M) \ge \frac{h^\eta(M)(h^\eta)'(M)}{8}. \end{equation}
 \end{theorem}
 
It is natural to ask whether the product $h^\eta(M)(h^\eta)'(M)$ is strictly positive in the case $\eta \not\in \mathfrak{B}_M$. We conjecture that this is the case but we do not have a proof at this point. The reason is that magnetic frustration constants are still a bit mysterious and it would be worthwhile to obtain a better understanding of this concept. Roughly speaking, a magnetic frustration constant measures how far a magnetic potential differs from one that can be gauged away. A magnetic frustration constant for a circle of length $L>0$ with constant magnetic potential was computed in \cite[Example 2.7]{ELMP:16} (under the name of frustration index). As a further step towards a better understanding, we compute magnetic frustration constants for two explicit $2$-dimensional examples in Subsection \ref{subsec:frustind}.

\begin{remark}
Due to the relationship between the magnetic Steklov and Neumann eigenvalue problems mentioned earlier, a straightforward adaptation of the proof of Theorem \ref{thm:magnjammes} above, given in Subsection \ref{subsec:lowerbdproof}, leads to an analogous result for the smallest magnetic Neumann eigenvalue, namely
$$ \lambda_{1,N}^\eta(M) \ge \frac{\left(h^\eta\right)^2}{8}. $$
\end{remark}

We also have an upper bound for the first magnetic Steklov eigenvalue which is analogous to the upper bound for the first magnetic Neumann eigenvalue in \cite[Thm. 1.2]{CSIS-21}. This is another application of the close relationship between the Steklov and the Neumann eigenvalue problems. 
 
For the result below, we need to recall a particular distance function given in \cite[(18)]{CSIS-21}. Let $c_1,\dots,c_l$ be a basis of $H_1(M,\mathbb{Z})$ with $l = b_1(M)$ and 
$$ \mathfrak{L}_{\mathbb{Z}} = \bigoplus_{k=1}^l \mathbb{Z} \omega_k $$
be a lattice in $H^1(M)$, where $\omega_1,\dots,\omega_l$ are defined via the relations
$$ \frac{1}{2\pi} \int_{c_j} \omega_k = \delta_{jk}. $$
For any harmonic $1$-form $\omega \in \Omega^1(M)$, we have the following distance to the lattice $\mathfrak{L}_{\mathbb{Z}}$:
$$ {\rm{dist}}(\omega,\mathfrak{L}_{\mathbb{Z}}) = \min \left\{ \Vert \omega - \beta \Vert_{L^2(M)} \mid \beta \in \mathfrak{L}_{\mathbb{Z}} \right\}. $$

\begin{theorem}[cf. {\cite[Thm 1.2]{CSIS-21}} for the Neumann eigenvalue version] \label{thm:colboissavo}
Let $(M^m,g)$ be a compact Riemannian manifold with smooth boundary $\partial M$ and let $\eta \in \Omega^1(M)$ be a magnetic potential. We can assume, without loss of generality, that 
$\eta$ is of the form
$$ \eta = \delta^M \eta_1 + \eta_2, $$
where $\eta_1 \in \Omega^2(M)$ satisfies $\nu \lrcorner \eta_1 = 0$ on $\partial M$ and $\eta_2 \in \Omega^1(M)$ satisfies $\nu \lrcorner \eta_2 = 0$ on $\partial M$ and $d^M \eta_2= \delta^M \eta_2 = 0$, since every magnetic potential is gauge-equivalent to such a potential by Remark \ref{rem:hodgedecomp} below.
Then we have
$$ \sigma_1^\eta(M) \le \frac{1}{|\partial M|} \left( {\rm{dist}}(\eta_2,\mathfrak{L}_{\mathbb{Z}})^2 + \frac{\Vert d^M\eta \Vert_{L^2(M)}^2}{\lambda_{1,1}''(M)}\right), $$
where
\begin{equation} \label{eq:lambda11''} 
\lambda_{1,1}''(M) := \inf \left\{ \frac{\Vert d^M \theta \Vert^2_{L^2(M)}}{\Vert \theta \Vert^2_{L^2(M)}}: \theta = \delta^M \beta \neq 0, \text{$\nu \lrcorner \beta =0$ on $\partial M$} \right\}. 
\end{equation}
\end{theorem}

\subsection{Full spectrum considerations for the magnetic Steklov operator}

After obtaining lower and upper bounds of the smallest Steklov eigenvalue in the previous subsection, we consider relations between the full spectrum of the magnetic Steklov operator $T^\eta$ on $(M,g)$ and the 
magnetic Laplacian $\Delta^{\eta_0}$ with $\eta_0 = \iota^* \eta$ on the boundary $\partial M$ with the induced metric. Note that, in the case $\eta \in \mathfrak{B}_M$, the magnetic Steklov spectrum reduces to the spectrum of the Steklov operator without magnetic potential and, in view of Remark \ref{rem:etaeta0relation}, the spectrum of $\Delta^{\eta_0}$ reduces to the spectrum of the non-magnetic Laplacian $\Delta^{\partial M}$.

Usually, these spectra cannot be explicitly computed. However, we present two examples for which this is possible. After completion of our computations, we realized that analogous computations for the following $2$-dimensional example have already been done in \cite{HN:24a}. Nevertheless, we include our computations for the reader's convenience. Moreover, in Subsection \ref{subsec:ex4dimball}, we also compute explicitly the magnetic Steklov spectrum of the $4$-dimensional Euclidean ball centered at the origin with magnetic potential $\eta = -y_1 dx_1 + x_1 dy_1 - y_2 dx_2 + x_2 dy_2$.

\begin{example}[Magnetic spectra for the disk] \label{ex:2disk}
Let $(\mathbb{B}^2,g)$ be the $2$-dimensional Euclidean unit ball, centered at the origin and $\mathbb{S}^1 = \partial \mathbb{B}^2$ and let $X = -y \partial_x + x \partial_y$ on $\mathbb{R}^2$. Then $X$ is a Killing vector field and its corresponding $1$-form 
$$ \eta = X^\flat = -y dx + x dy, $$  
restricted to the punctured disk $\mathbb{B}^2 \setminus \{0\}$, coincides with $r^2 d\theta$, where $r = \sqrt{x^2+y^2}$ and $\theta: \mathbb{B}^2 \setminus \{0\} \to (\mathbb{R} \mod 2\pi)$ is the angle map, and the spectrum of the magnetic Steklov operator $T^{t \eta}$ ($t > 0$) is given by
$$
{\rm{spec}}(T^{t\eta}) = \left\{ \sigma^\pm(k) = 
\frac{-(k\pm t +1)L_{-1/2}^{(-k)}(\pm t) + (1 + 2k)L_{-3/2}^{(-k)}(\pm t)}{L_{-1/2}^{(-k)}(\pm t)}: k \in \mathbb{N} \cup \{0\} \right\},
$$
where $L_p^{(\alpha)}(\pm t)$ are the generalized Laguerre polynomials. The $t$-dependence of the spectrum of $T^{t \eta}$ is illustrated in Figure \ref{fig:spectrum-comparison}(left). The Steklov eigenfunction corresponding to $\sigma^\pm(k)$ is $e^{\pm i k \theta}$. We note that Laguerre functions also appear in the description of the eigenfunctions and eigenvalues of the magnetic Neumann Laplacian $\Delta^{\beta \eta/2}: C^\infty(\mathbb{B}^2(R)) \to C^\infty(\mathbb{B}^2(R))$ on the Euclidean ball $\mathbb{B}^2(R)$ of radius $R > 0$, centered at the origin, for $\beta > 0$ (see \cite[Appendix B]{CLPS:23}). It is well known (see, e.g., \cite[Example 2.3]{ELMP:16}) that the eigenvalues of $\Delta^{t \eta_0}: C^\infty(\mathbb{S}^1) \to C^\infty(\mathbb{S}^1)$ with $\eta_0 = \iota^* \eta = d\theta \in \Omega^1(\mathbb{S}^1)$ are given by 
$$ {\rm{spec}}(\Delta^{t \eta_0}) = \{ \lambda^\pm(k) = (k \pm t)^2: k \in \mathbb{N} \cup \{0\} \}. $$
The $t$-dependence of the spectrum of $\Delta^{t \eta_0}$
is illustrated in Figure \ref{fig:spectrum-comparison}(right). The Laplace eigenfunction corresponding to $\lambda^{\pm}(k)$ is again $e^{\pm i k \theta}$. While in Figure \ref{fig:spectrum-comparison} both smallest eigenvalues $\sigma_1^{t\eta}(\mathbb{D}^2)$ and $\lambda_1^{t\eta_0}(\mathbb{S}^1)$ exhibit an oscillating behavior as functions of $t$, the first one appears to be increasing as $|t| \to \infty$, while the second one is periodic in $t$, and therefore remains bounded. We previously conjectured that $\sigma_1^{t \eta}(\mathbb{B}^2) \to \infty$ as $t \to \infty$. In fact, the main result of recent work by Helffer and Nicoleau proves this conjecture and provides the following asymptotic expansion (see \cite[Theorems 1.1]{HN:24b}):
$$ \sigma_1^{t \eta}(\mathbb{B}^2) = \alpha t^{1/2} - \frac{\alpha^2+2}{6} + O(t^{-1/2}) \quad \text{as $t \to \infty$,} $$
with a constant $\alpha\approx 0.7649508693$. 
Moreover, together with Kachmar, they obtain an analogous asymptotic expansion for arbitrary regular domains $\Omega \subset \mathbb{R}^2$ (see \cite[Theorem 1.2]{HKN:25}).
In addition, $t \mapsto \sigma_1^{t \eta}(\mathbb{B}^2)$ is increasing on $(0,\infty)$ (see \cite[Theorem 1.2]{HN:24b}). 
This example shows that there cannot exist a finite constant $C > 0$ such that 
$$ | \sigma_1^{t\eta}(\mathbb{B}^2) - \sqrt{\lambda_1^{t\eta_0}(\mathbb{S}^1)} | \le C \quad \text{for all $t \ge 0$.} $$

For comparison, note the following work regarding the magnetic Neumann eigenvalues on $\mathbb{B}^2$:
It follows from \cite[Theorem 2.5]{CLPS:23} that the first Neumann eigenvalue satisfies $\lambda_{1,N}^{t\eta}(\mathbb{B}^2) \ge C t$ for some $C>0$ and sufficiently large $t$, which implies that $\lambda_{1,N}^{t\eta}(\mathbb{B}^2) \to \infty$ as $t \to \infty$. Moreover, other related results and interesting conjectures are given in the recent paper \cite{HL:24}.

\begin{figure}
     \begin{center}
     \includegraphics[width=0.48\textwidth]{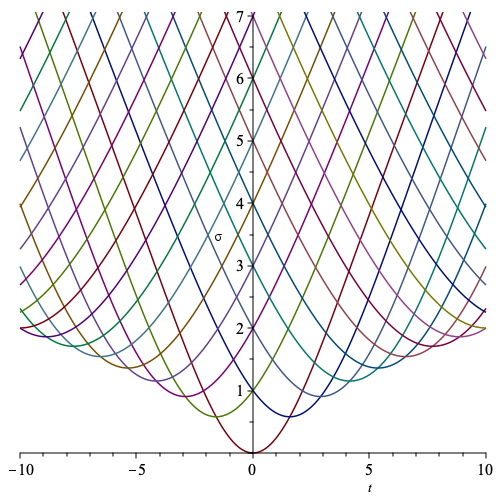}
     \includegraphics[width=0.48\textwidth]{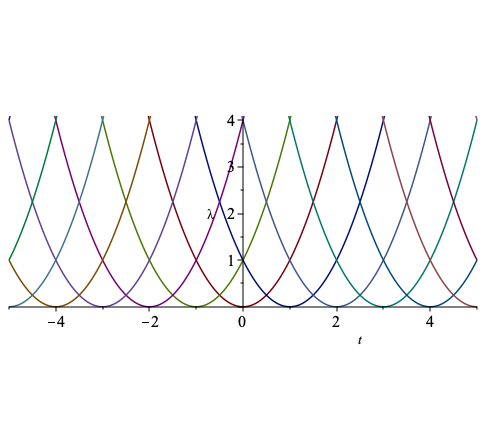}     
     \end{center}
     \caption{Eigenvalues of the magnetic Steklov operator $T^{t \eta}$ on $\partial \mathbb{B}^2$ as functions in $t$ (left) and of the magnetic Laplacian $\Delta^{t \eta_0}$ on $\mathbb{S}^1$ as functions in $t$ (right).
     \label{fig:spectrum-comparison}}
\end{figure}

\end{example}

We end our survey of results by the following comparison result of the sequences $\sigma_k^\eta(M)$ and $\sqrt{\lambda_k^{\eta_0}(\partial M)}$. They are based on results of \cite{PS:19} for Euclidean domains with $C^2$-boundary and \cite{CGH:20} for Riemannian manifolds (see also \cite{GKLP:22} for an analogous result for the Steklov problem on differential forms, as formulated in \cite{K:19}). Before we present this result, we introduce the following notation. For a compact manifold $(M,g)$ with smooth boundary $\partial M$, we denote by ${\rm{roll}}(M)$ the rolling radius of $M$, that is, the distance between $\partial M$ and its cut-locus. For any $h \in (0,{\rm{roll}}(M))$ fixed, we denote by $\Sigma_h$ the parallel hypersurface at distance $h$ from $\partial M$ and by $M_h := \{ x \in M: {\rm{dist}}(x,\partial M) < h \}$ the tubular neighborhood of $\partial M$.  Let $\gamma: M_h \to [0,\infty)$ be the function
$$ \gamma(x) := \frac{1}{2} {\rm{dist}}(x,\Sigma_h)^2. $$
Note that $\grad \gamma$ is normal to the parallel hypersurfaces in $M_h$ and $\grad \gamma = - h \nu$ on $\partial M$.

\begin{theorem}[Magnetic version of {\cite[Thm. 3]{CGH:20}}]\label{thm:intro_magnetic steklov-laplacian comparison}
Let $(M,g)$ be a compact Riemannian manifold with smooth boundary $\partial M$ and $\eta \in \Omega^1(M)$ be a magnetic potential such that $\eta \not\in \mathfrak{B}_M$. Let $\eta_0 = \iota^* \eta \in \Omega^1(\partial M)$. Then there exist $\tilde h \in (0,{\rm{roll}}(M))$, and constants $C_1 >0$, depending only on bounds on the sectional curvature of $M_h$ and the principal curvatures of $\pam$, and $C_2 > 0$, depending on the geometry of the whole manifold $M$, such that we have for any $0 < h < \tilde h$: 
$$ \left| \sigma_k^\eta(M)- \sqrt{\lambda_k^{\eta_0}(\partial M)}  \right| \le C_1 + C_2 \frac{\Vert \grad \gamma \lrcorner d^M\eta\Vert_{M_h,\infty}}{\sqrt{\sigma_1^\eta(M)}}. $$
\end{theorem}

\begin{remark}
    The theorem implies that, for any fixed compact manifold with boundary and fixed magnetic potential, corresponding eigenvalues of the two spectra can be compared in a uniform way for all eigenvalues. For a more detailed version of the result, see Theorem \ref{thm:magnetic steklov-laplacian comparison} in Section \ref{sec:speccomparison}. Moreover, the first Steklov eigenvalue in the upper bound can be estimated by the Cheeger constants via Theorem \ref{thm:magnjammes}, in case they are positive.
    
    Moreover, in the special case of $\grad \gamma \lrcorner d^M\eta = 0$ on $M_h$ (that is, if the magnetic field $d^M \eta$ ``does not have radial dependence near $\partial M$''), the estimate gives back similar bounds to those in \cite[Thm. 3]{CGH:20} for the case without magnetic potential. This means, in particular, that in this case the estimate is completely independent of the strength of the magnetic field $d^M \eta$. 
\end{remark}

\section{Proof of the Maximum Principle and Theorem \ref{thm:shikegawa}} 

\label{sec:maxprincandshike}

The magnetic Maximum Principle in Proposition \ref{prop:maxprinc} is a special case of the following more general result, which might be folklore, but we provide a proof the reader's convenience (see also \cite{AS:14}). 

\begin{proposition} \label{prop:maxprincgen} Let $(M^m,g)$ be a Riemannian manifold and $\pi: E \to M$ be a complex vector bundle over $M$ equipped with a Hermitian metric $\langle \cdot,\cdot \rangle_E$ and a Riemannian connection $\nabla^E$, that is
$$ X \langle s_1,s_2 \rangle_E = \langle \nabla^E_X s_1,s_2 \rangle_E + \langle s_1, \nabla^E_X s_2 \rangle_E. $$
Let $\Delta^E$ be the connection Laplacian acting on sections on $E$, that is, if $\{e_k\}_{k=1,\ldots,m}$ is a local orthonormal frame of $TM$ with $\nabla e_k = 0$ at $x \in M$, then
$$ (\Delta^E s)(x) = - \left( \sum_{k=1}^m \nabla^E_{e_k}
\nabla^E_{e_k} s\right)(x). $$
If there exists a section $s\in \Gamma(E)$ such that $\Delta^E s=0$, then the function $|s|^2$ is a nonnegative and subharmonic function, that is, 
$$\Delta^M(|s|^2)\leq 0.$$  
Hence $\mathop{\rm max}\limits_{\partial M}|s|=\mathop{\rm max}\limits_{M}|s|$. 

In the particular case when $E=\Omega^p(M,\mathbb{C})$, $p\in\{0,\ldots, m\}$, $\eta$ is a magnetic potential on $M$, and $\Delta^\eta$ is the magnetic Hodge Laplacian for differential forms, that is, 
$$
\Delta^\eta := d^\eta \delta^\eta + \delta^\eta d^\eta
$$
with the magnetic differential $d^\eta \omega := d^M \omega + i \eta \wedge \omega$ and the magnetic codifferential $\delta^\eta \omega := \delta^M \omega - i \eta \lrcorner \omega$, we have the following: if the magnetic Bochner operator $\mathcal{B }^{[p],\eta}$, defined in \cite[Section 3.2]{EGHP:23}, satisfies $\mathcal{B }^{[p],\eta}\geq 0$, then any $\eta$-harmonic differential $p$-form $\omega$ satisfies $\Delta^M(|\omega|^2)\leq 0$. 
\end{proposition}

\begin{remark}
Proposition \ref{prop:maxprinc} follows from Proposition \ref{prop:maxprincgen} by choosing $E = M \times \mathbb{C}$ and 
$$ \nabla_v^E f = d^\eta f(v) = v(f) + i f \eta(v)  $$ 
for $f \in C^\infty(M)$ and $v \in TM$.
\end{remark}

\begin{proof}[Proof of Proposition \ref{prop:maxprincgen}]
To compute the Laplacian of the function $|s|^2$, we take an orthonormal frame $\{e_k\}_{k=1,\ldots,m}$ of $TM$ such that $\nabla^M e_k=0$ at $x\in M$. We have
\begin{eqnarray}\label{eq:subharmonicinequality}
    \Delta^M (|s|^2)&=&-\sum_{k=1}^m e_k(e_k(|s|^2))\nonumber\\
    &=&-\sum_{k=1}^m e_k(\langle\nabla^E_{e_k}s, s\rangle)+\langle s,\nabla^E_{e_k}s\rangle)\nonumber\\
    &=&\langle \Delta^Es,s\rangle-2|\nabla^E s|^2+\langle s,\Delta^Es\rangle\\
    &=&-2|\nabla^Es|^2\leq 0.\nonumber
\end{eqnarray}
Hence, the function $|s|^2$ is subharmonic. When $E=\Omega^p(M,\mathbb{C})$ for $p\in \{0,\ldots,m\}$ and $\eta\in \Omega^1(M)$ is the magnetic potential, we take the connection $\nabla_X^E:=\nabla_X^M+i\eta(X)$, for all $X\in TM$. Hence from \eqref{eq:subharmonicinequality} and the magnetic Bochner-Weitzenb\"ock formula  $\Delta^\eta=\Delta^E+\mathcal{B }^{[p],\eta}$ which is valid on $p$-forms \cite[Thm. 3.4]{EGHP:23}, we get for any $\eta$-harmonic differential $p$-form $\omega$, the following
$$\Delta^M (|\omega|^2)=-2\langle \mathcal{B }^{[p],\eta}\omega,\omega\rangle-2|\nabla \omega|^2\leq 0,$$
if the curvature term satisfies $\mathcal{B }^{[p],\eta}\geq 0$. This finishes the proof.
\end{proof}

Later, we will also need an inequality for subharmonic functions on a compact Riemannian manifold $M$ with smooth boundary, relating their $L^2$-norms over $M$ and $\partial M$ and given in Proposition \ref{prop:L2estimate} below. To that end, let us first introduce the relevant ingredients. Let $K$ and $H_0$ be real numbers. We say that $M$ has curvature bounds $(K,H_0)$ if 
$${\rm Ric}_M\geq (m-1)K \quad\text{and}\quad H\geq H_0,$$
where $H=\frac{1}{m-1}{\rm trace}({\bf II})$ is the mean curvature of $\partial M$. Here ${\bf II}=-\nabla^M\nu$ is the second fundamental form of $\partial M$. We denote by  $R$ be the inner radius of $M$, that is 
$$R={\rm max}\{{\rm dist}(x,\partial M),\,\,  x\in M\}.$$ 
Let $s_K(r)$ be the function given by 
\begin{equation}\label{eq:skh}
s_K(r)=\left\{
\begin{matrix}
 \frac{1}{\sqrt{K}} {\rm sin}(\sqrt{K}r)&  \textrm{if}\,\, K>0\\\\
r\,\,  &\textrm{if}\,\, K=0\\\\
\frac{1}{\sqrt{|K|}} {\rm sinh}(\sqrt{|K|}r) &  \textrm{if}\,\, K<0\\\\
\end{matrix}\right.
\end{equation}
and consider the function $\Theta(r):=(s'_K(r)-H_0s_K(r))^{m-1}$. It is shown in \cite[Thm. A]{Ka} (see also \cite[Prop. 14]{RS:15}) that if $M$ has curvature bounds $(K,H_0)$, the function $\Theta$ is smooth and positive on $[0, R)$ and that $\Theta(R)=0$ if and only if $M$ is a ball in the simply connected manifold of constant curvature $K$. Based on the definitions, we recall now the following result in \cite[Thm. 10]{RS:15}. 
\begin{theorem} \label{thm: subhar} Let $(M^m,g)$ be a compact Riemannian manifold with smooth boundary. Assume that $M$ has curvature bounds $(K,H_0)$. If $h$ is a nontrivial, smooth, nonnegative and subharmonic function (i.e. $h\geq 0$ and $\Delta h\leq 0$) on $M$, then 
$$\frac{\int_{\partial M} h{\rm dvol}_{\partial M}}{\int_{M} h {\rm dvol}_M}\geq \frac{1}{\int_0^R\Theta(r) dr}.$$
\end{theorem}
As a direct consequence of Proposition \ref{prop:maxprinc} and Theorem \ref{thm: subhar}, we deduce the following result for $\eta$-harmonic $p$-differential forms. 

\begin{proposition}\label{prop:L2estimate}
Let $(M^m,g)$ be a compact Riemannian manifold with smooth boundary and let $\eta\in \Omega^1(M)$ be the magnetic potential. Let $p\in \{0,\ldots,m\}$. Assume that $M$ has curvature bounds $(K,H_0)$ and that  $\mathcal{B }^{[p],\eta}\geq 0$.  If $\omega$ is an $\eta$-harmonic $p$-differential form, then 
\begin{equation*}
\int_{M} |\omega|^2 {\rm dvol}_M\leq \left(\int_0^R\Theta(r) dr\right) \int_{\partial M} |\omega|^2{\rm dvol}_{\partial M}.
\end{equation*}
\end{proposition}

Now we provide the proof of Theorem \ref{thm:shikegawa}.

\begin{proof}[Proof of Theorem \ref{thm:shikegawa}]
We first prove equivalence of (i) and (ii). The equivalence of (ii) and (iii) is already established in \cite[Prop. 3.1]{Shi:87}, and unitary equivalence of $\Delta^\eta$ and $\Delta^{\eta+\alpha}$ for $\alpha \in \mathfrak{B}_M$, used below, is established in \cite[Prop. 3.2]{Shi:87}.

Assume first that $\sigma_1^\eta(M)=0$. Hence from the characterization \eqref{eq:characsigma}, we get that $d^\eta \hat f=0$, where $\hat f$ is the $\eta$-harmonic extension of the complex eigenfunction $f$ that corresponds to the smallest eigenvalue of the Steklov operator. Thus, we get that $d^M\hat f=-i\hat f\eta$. Therefore, we compute 
$$d^M(|\hat f|^2)=\hat f d^M(\overline{\hat f})+\overline{\hat f}d^M\hat f=i|\hat f|^2\eta-i|\hat f|^2\eta=0.$$
Hence $|\hat f|^2=c^2$ is a non-zero constant, since otherwise $f$ would be zero. By taking $\tau=\frac{\overline{\hat f}}{c}$, we get that $\eta=\frac{d^M\tau}{i\tau}$. Thus $\eta\in \mathfrak{B}_M$. For the converse, assume now $\eta\in \mathfrak{B}_M$ and let $f$ be an eigenfunction of the non-magnetic Steklov operator associated with an eigenvalue $\sigma(M)$. By writing $\eta=\frac{d^M\tau}{i\tau}$ for some function $\tau\in C^\infty(M,\mathbb{S}^1)$, we set $\hat{h}:=\overline\tau \hat f$, where $\hat f$ is the harmonic extension of $f$. By using the fact that, when $\eta\in \mathfrak{B}_M$, the operators $\Delta^M$ and $\Delta^\eta$ are unitarily equivalent, meaning that  $\overline\tau\Delta^M\tau=\Delta^\eta$, we get that $\Delta^\eta \hat h=\overline\tau\Delta^M\hat f=0$. Hence, $\hat h$ is the $\eta$-harmonic extension of the function $h=\overline{\tau}f$. Moreover, we compute 
\begin{eqnarray*}
    \nu\lrcorner d^\eta \hat h&=&\nu\lrcorner (d^M\hat h+i\hat h\eta)\\
    &=&\nu\lrcorner(d^M(\overline\tau \hat f)+i\overline\tau \hat f\eta)\\
    &=&\nu\lrcorner(-i\overline\tau \hat f\eta+\overline\tau d^M\hat f+i\overline\tau \hat f\eta)\\
    &=&\sigma(M)h.
\end{eqnarray*}
Hence, we deduce that $h$ is an eigenfunction of the magnetic Steklov operator associated with the same eigenvalue $\sigma(M)$. That means, the spectrum of the Steklov operator is a subset of the magnetic one. In the same way, we prove the other inclusion. This gives that both operators have the same spectrum and, thus, we deduce $\sigma_1^\eta(M)=\sigma_1(M)=0$. This finishes the proof. 
\end{proof}

\section{Proofs of the lower and upper magnetic Steklov eigenvalue bounds}
In this section we prove Theorem \ref{thm:magnjammes} and Theorem \ref{thm:colboissavo}.
\subsection{Proof of Theorem \texorpdfstring{\ref{thm:magnjammes}}{2.5} (lower bound)}
\label{subsec:lowerbdproof}

As mentioned earlier, the Cheeger constant and its associated spectral inequality were originally introduced by Cheeger in \cite{Ch:70}. Magnetic Cheeger constants are modifications of the original one by incorporating a contribution from the magnetic potential $\eta$ via a magnetic frustration constant, given in Definition \ref{def:frustindex}. Our Theorem \ref{thm:magnjammes} is a magnetic version of a result by Jammes \cite{Ja:15} and is proved first.

\begin{proof}[Proof of Theorem \ref{thm:magnjammes}]
  Let $f \in C^\infty(\partial M,\mathbb{C})$ be an eigenfunction of $T^\eta$ associated to the eigenvalue $\sigma_1^\eta(M)$ and
  ${\hat f} \in C^\infty(M,\mathbb{C})$ be its $\eta$-harmonic extension. We compute
  \begin{eqnarray}
      \sigma_1^\eta(M) &=& \frac{\int_M |d^\eta {\hat f}|^2{\rm dvol}_M}{\int_{\partial M} | f|^2 {\rm dvol}_M} = \frac{(\int_M |{\hat f}|^2{\rm dvol}_M)(\int_M |d^\eta {\hat f}|^2 {\rm dvol}_M)}{(\int_M |{\hat f}|^2{\rm dvol}_M)(\int_{\partial M} | f|^2{\rm dvol}_{\partial M}) } \nonumber \\ 
      &\ge& \frac{(\int_M |{\hat f}| \cdot |d^\eta {\hat f}|{\rm dvol}_M)^2}{(\int_M |{\hat f}|^2{\rm dvol}_M)(\int_{\partial M} | f|^2{\rm dvol}_{\partial M})}\nonumber\\
      &=& \frac{(\int_M |d^{2\eta} {\hat f}^2|{\rm dvol}_M)^2}{4(\int_M |{\hat f}|^2{\rm dvol}_M)(\int_{\partial M} |f|^2{\rm dvol}_M)} \nonumber \\ 
      &=& \frac{1}{4} \left(\frac{\int_M |d^{2\eta} {\hat f}^2|{\rm dvol}_M}{\int_M |{\hat f}|^2{\rm dvol}_M}\right) \left(\frac{\int_M |d^{2\eta} {\hat f}^2|{\rm dvol}_M}{\int_{\partial M} |f|^2{\rm dvol}_{\partial M}}\right). \label{eq:cheegest}
  \end{eqnarray}
  In the estimate above, we used the Cauchy-Schwarz inequality. Let $\hat f_0 = \vert \hat f \vert^2$. By the co-area formula, we have
  $$ \int_M |d{\hat f}_0|{\rm dvol}_M = \int_0^\infty {\rm{area}}({\hat f}_0^{-1}(t)) dt \ge \int_0^\infty {\rm{area}}(\underbrace{\partial \Omega_{\tilde f}(\sqrt{t}) \cap {\rm{int}}(M)}_{= \partial_I \Omega_{\hat f}(\sqrt{t})}) ds $$
  where $\Omega_{\hat f}(s) = |{\hat f}|^{-1}([s,\infty))$. In the following, we can assume without loss of generality that $f_0(x)>0$ for any $x\in M$ (since otherwise we integrate over $\Omega_{\hat f}(\epsilon)$ and then we let $\epsilon\to 0$), and we write ${\hat f}(x) = \left(\sqrt{{\hat f}_0(x)}\right) \tau$ with $\tau: M \to S^1 \subset \mathbb{C}$. Let $\alpha_{\hat f} = \frac{d^M \tau}{i \tau}$. Then we have, by a straightforward computation,
  $$ |d^{2\eta}{\hat f}^2| = \sqrt{|d{\hat f}_0|^2 + |{\hat f}_0(2\eta+2\alpha_{\hat f})|^2} \ge \frac{1}{\sqrt{2}} (|d{\hat f}_0|+|{\hat f}_0(2\eta+2\alpha_{\hat f})|). $$
  By integration over $M$, we obtain
  \begin{align*}
  \int_M |{\hat f}_0(2\eta+2\alpha_{\hat f})| {\rm dvol}_M&= 2 \int_0^\infty \int_{\hat f_0^{-1}([t,\infty)]} |\eta+\alpha_{\hat f}| {\rm dvol}_M dt\\
  &= 2\int_0^\infty \int_{\Omega_{\hat f}(\sqrt{t})} |\alpha_{\hat f} +  \eta| {\rm dvol}_M dt. 
  \end{align*}
  This implies that
  \begin{eqnarray*} 
 \int_M|d^{2\eta}{\hat f}^2| {\rm dvol}_M &\ge& \frac{1}{\sqrt{2}} \int_0^\infty \left( {\rm{area}}(\partial_I \Omega_{\hat f}(\sqrt{t})) + 2 \int_{\Omega_{\hat f}(\sqrt{t})} |\alpha_{\hat f}+\eta| {\rm dvol}_M\right) dt\\
  &\ge& \frac{1}{\sqrt{2}} \int_0^\infty \left( {\rm{area}}(\partial_I \Omega_{\hat f}(\sqrt{t})) + \iota^\eta(\Omega_{\hat f}(\sqrt{t}))\right) dt,
  \end{eqnarray*}
  since $\alpha_{\hat f} \in \mathfrak{B}_{\Omega_{\hat f}(\sqrt{t})}$.
  From this we obtain
  \begin{eqnarray*} 
  \frac{\int_M|d^{2\eta}{\hat f}^2| {\rm dvol}_M}{\int_M |{\hat f}|^2{\rm dvol}_M} &\ge& \frac{1}{\sqrt{2}}\frac{\int_0^\infty \left( {\rm{area}}(\partial_I\Omega_{\hat f}(\sqrt{t})) + \iota^\eta(\Omega_{\hat f}(\sqrt{t}))\right) dt}{\int_0^\infty {\rm{vol}}(\Omega_{\hat f}(\sqrt{t})) dt} \\
&\ge& \frac{h^\eta(M)}{\sqrt{2}}
  \end{eqnarray*}
  on the one hand, and
  \begin{eqnarray*} 
  \frac{\int_M|d^{2\eta}{\hat f}^2| {\rm dvol}_M}{\int_{\partial M} |f|^2{\rm dvol}_{\partial M}} &\ge& \frac{1}{\sqrt{2}}\frac{\int_0^\infty \left( {\rm{area}}(\partial_I \Omega_{\hat f}(\sqrt{t})) + \iota^\eta(\Omega_{\hat f}(\sqrt{t}))\right) dt}{\int_0^\infty {\rm{area}}(\Omega_{\hat f}(\sqrt{t})\cap \partial M) dt} \\
  &=& \frac{1}{\sqrt{2}}\frac{\int_0^\infty \left( {\rm{area}}(\partial_I \Omega_{\hat f}(\sqrt{t})) + \iota^\eta(\Omega_{\hat f}(\sqrt{t}))\right) dt}{\int_0^\infty |\partial_E \Omega_{\hat f}(\sqrt{t}))| dt} \\
  &\ge& \frac{(h^\eta)'(M)}{\sqrt{2}}.
  \end{eqnarray*}
  on the other hand. Plugging these inequalities into \eqref{eq:cheegest}, we obtain the desired Cheeger inequality.
\end{proof}

\begin{remark}
    An analogous proof leads to a corresponding result for the magnetic Neumann problem, namely, we have the estimate
    $$ \lambda_{1,N}^\eta(M) \ge \frac{\left(h^\eta(M)\right)^2}{8}. $$
    A corresponding Cheeger inequality for closed manifolds can be found in \cite[Thm. 7.4]{LLPP:15}.
\end{remark}

A natural question about Theorem \ref{thm:magnjammes} is whether the magnetic Cheeger constants $h^\eta(M)$ and $(h^\eta)'(M)$ are strictly positive in the case of $\eta \not\in \mathfrak{B}_M$. In this case, we have $\sigma_1^\eta(M) > 0$, and the inequality \eqref{eq:cheegerjammesineq} would become meaningless if we had $h^\eta(M)(h^\eta)'(M)=0$. For the same reason, non-vanishing of the classical Cheeger constant was also discussed and confirmed in \cite{Ch:70} and \cite[Thm. 7.1]{Bu:82}. While we cannot prove this in the magnetic case, it would be a consequence of the following Conjecture.

\begin{conjecture} Let $(M^m,g)$ be a compact manifold with smooth boundary $\partial M$ and let $\eta \in \Omega^1(M)$. If $\eta \not\in\mathfrak{B}_M$, then there exist $\epsilon,\delta >0$, with the following property. If $D \subset M$ is a non-empty open subset with compact closure whose boundary $\partial D$ is a smooth $(m-1)$-dimensional submanifold and
$$ |D| \ge (1-\epsilon)|M| \quad \text{and} \quad |\partial_I D| \le \epsilon, $$
then
$$ \iota^\eta(D) \ge \delta. $$
\end{conjecture}

Note that the facts in \eqref{eq:iotarldtheta} and \eqref{eq:iotadthetapunc} of the next subsection are not a contradiction to the conjecture, since $d\theta$ cannot be extended smoothly from $\mathbb{B}^2(r_0)\setminus\{0\}$ to $\mathbb{B}^2(r_0)$ (see Remark \ref{rem:etaeta0relation}). 

Let us briefly explain why this conjecture would imply $h^\eta(M)(h^\eta)'(M)>0$ for $\eta \not\in\mathfrak{B}_M$. Firstly, we show $h^\eta(M) > 0$. For this we distinguish the following cases:
\begin{itemize}
\item[(a)] If $|D| \le \frac{|M|}{2}$, we have
$$ \frac{\iota^\eta(D) + |\partial_I D|}{|D|} \ge \frac{|\partial_I D|}{|D|} \ge h(M), $$
where $h(M)$ is the classical Cheeger constant for compact manifolds with boundary, which is known to be $>0$.
\item[(b)] If $\frac{1}{2}|M| < |D| < (1-\epsilon)|M|$, then $|M\setminus D| < \frac{|M|}{2}$, $\partial_I D = \partial_I(M\setminus D)$ and
$$ \frac{\iota^\eta(D) + |\partial_I D|}{|D|} \ge \frac{|\partial_I (M \setminus D)|}{|M \setminus D|}\cdot \frac{|M \setminus D|}{|D|} > \frac{\epsilon}{1-\epsilon} h(M) > 0. $$
\item[(c)] If $|\partial_I D| > \epsilon$, then
$$ \frac{\iota^\eta(D) + |\partial_I D|}{|D|} \ge \frac{\epsilon}{|M|} > 0. $$
\item[(d)] If $|D| \ge (1-\epsilon)|M|$ and $|\partial_I D| \le \epsilon$, then, by the conjecture, we would have
$$ \frac{\iota^\eta(D) + |\partial_I D|}{|D|} \ge \frac{\delta}{|M|} > 0.$$
\end{itemize}
This confirms $h^\eta(M) > 0$. To show $(h^\eta)'(M) > 0$, we distinguish two cases:
\begin{itemize}
\item[(a)] If $|D| \le \frac{|M|}{2}$, we have
$$ \frac{\iota^\eta(D) + |\partial_I D|}{|\partial_E D|} \ge \frac{|\partial_I D|}{|\partial_E D|} \ge h'(M), $$
with $h'(M)$ defined in \cite[(4)]{Ja:15} satisfying $h'(M) > 0$ by \cite[Thm. 1]{Ja:15}.
\item[(b)] If $|D| > \frac{|M|}{2}$, we have
$$ \frac{\iota^\eta(D) + |\partial_I D|}{|\partial_E D|} \ge \frac{\iota^\eta(D)+|\partial_I D|}{|D|} \cdot \frac{|D|}{|\partial_E D|} \ge h^\eta(M) \cdot \frac{|M|}{2|\partial M|} > 0. $$
\end{itemize}
This completes the proof of $h^\eta(M)(h^\eta)'(M) > 0$, if the above conjecture is true.

\subsection{Examples of magnetic frustration constants}
\label{subsec:frustind}

To become a little bit more familiar with magnetic frustration constants, we now compute them for two particular $2$-dimensional examples.

\begin{example}[Magnetic frustration constant for $\eta = r^k dr + r^\ell d\theta$ on the disk] Let $D = \mathbb{B}^2(r_0) \subset \mathbb{R}^2$ be the $2$-dimensional Euclidean ball of radius $r_0 > 0$ centered at the origin and 
$$ \eta = r^k dr + r^\ell d\theta $$
for integers $k, \ell \in \mathbb{N}$. To guarantee that $\eta$ is also well-defined at the origin, we need to assume that $k \ge 1$ and $\ell \ge 2$. Our aim is to compute 
$$ \iota^\eta(\mathbb{B}^2(r_0)) = \inf_{\alpha \in \mathfrak{B}_D} \int_D | \eta + \alpha | d{\rm{vol}}_D.
$$
Note that 
$$ r^k dr = d^D \left( \frac{r^{k+1}}{k+1}\right) $$
is exact and therefore in $\mathfrak{B}_{D}$. Therefore, we have
$$ \iota^\eta(D) = \inf_{\alpha \in \mathfrak{B}_D} \int_D | r^\ell d\theta + \alpha| d{\rm{vol}}_D = \int_0^{r_0} \int_{S_r} | r^\ell d\theta + \alpha| d{\rm{vol}}_{S_r} dr, $$
where $S_r \subset \mathbb{R}^2$ is the circle of radius $r$ around the origin. Any $\alpha \in \mathfrak{B}_{D}$ can be written as
$$ \alpha = f dr + g d \theta $$
with suitable functions $f,g$. Since $\alpha \in \mathfrak{B}_{D}$, we have $d^D \alpha = 0$ and, by the Poincar\'e Lemma, $\alpha$ is exact. We conclude that
$$ \int_{S_r} |r^\ell d\theta + \alpha| d{\rm{vol}}_{S_r} \ge \left\vert \int_{S_r} r^\ell d\theta + \alpha  
\right\vert \stackrel{(*)}{=} r^\ell \left\vert \int_{S_r} d\theta \right\vert = 2\pi r^\ell, $$
where we used in $(*)$ the fact that the integration of an exact $1$-form along a closed line integral vanishes. Consequently, we have
$$ \iota^\eta(D) \ge \int_0^{r_0} 2 \pi r^\ell dr = \frac{2 \pi r_0^{\ell+1}}{\ell+1}. $$
On the other hand, we have $| r d\theta | = 1$, and therefore
$$ \iota^\eta(D) \le \int_D |r^\ell d\theta| d{\rm{vol}}_D = 2 \pi \int_0^{r_0} r |r^\ell d\theta| dr = \frac{2 \pi r_0^{\ell+1}}{\ell+1}, $$
which shows that we have for $k,\ell \in \mathbb{N}$, $k \ge 1$, $\ell \ge 2$,
\begin{equation} \label{eq:iotarldtheta}
\iota^{r^k dr + r^\ell d\theta}(\mathbb{B}^2(r_0)) = \frac{2\pi r_0^{\ell+1}}{\ell+1}. 
\end{equation}
\end{example}

In contrast to our first example, the following example has a non-trivial topology by the removal of the origin.

\begin{example}[Magnetic frustration constant for $\eta = r^k dr + r^\ell d\theta$ on the punctured disk] Let $D = \mathbb{B}^2(r_0) \setminus \{0\} \subset \mathbb{R}^2$ and 
$$ \eta = r^k dr + r^\ell d\theta $$
for integers $k, \ell \in \mathbb{Z}$. In contrast to the previous example, we do not need to impose any restrictions on $k, \ell$, since the origin is removed. Similarly, as before, we can drop the term $r^k dr$, since it is exact on $D$:
$$ r^k dr = \begin{cases} d^D \left( \frac{r^{k+1}}{k+1}\right), & \text{if $k \neq -1$,}\\ d^D \log(r), & \text{if $k=-1$.} \end{cases}$$
Using analogous arguments as in the previous example, we have, for $\alpha \in \mathfrak{B}_D$,
\begin{align*} 
\int_D |r^\ell d\theta + \alpha| d{\rm{vol}}_D &= \int_0^{r_0} \int_{S_r} |r^\ell d\theta + \alpha | d{\rm{vol}}_{S_r} dr\\ &\ge
\int_0^{r_0} \left\vert r^\ell \int_{S_r} d\theta + \int_{S_r} \alpha \right\vert dr \\
& = \int_0^{r_0} | 2 \pi r^\ell + 2 m \pi | dr
\end{align*}
with a suitable integer $m \in \mathbb{Z}$, since $\int_C \alpha \in 2 \pi \mathbb{Z}$ for all closed curves $C$ and since the map $r \mapsto \int_{S_r} \alpha$ is a continuous map. Moreover, for given $m \in \mathbb{Z}$, by choosing $\alpha = m d\theta \in \mathfrak{B}_D$, we can realise $\int_{S_r} \alpha = m$ for all $r \in (0,r_0)$. With arguments as before, for all $k,\ell \in \mathbb{Z}$, we end up with the result
$$ \iota^{r^k dr + r^\ell d\theta}(\mathbb{B}^2(r_0)\setminus \{0\}) = 2 \pi \min_{m \in \mathbb{Z}} \int_0^{r_0} |r^\ell + m| dr. $$
In particular, we have for $\ell = 0$, that is  
$$ \eta = d \theta = \frac{- y dx + x dy}{x^2+y^2}, $$
that
\begin{equation} \label{eq:iotadthetapunc}
\iota^{d\theta}(\mathbb{B}^2(r_0)\setminus \{0\}) = 0 
\end{equation}
and for $\ell \in \mathbb{N}$ and $r_0 =1$ that
$$ \iota^{r^\ell d\theta}(\mathbb{B}^2(1)\setminus \{0\}) = \frac{2\pi}{\ell+1}. $$
\end{example}

\subsection{Proof of Theorem \texorpdfstring{\ref{thm:colboissavo}}{2.7} (upper bound)}

In the second part of this section, we prove the upper bound for $\sigma_1^\eta(M)$ given in Theorem \ref{thm:colboissavo}. In the proof, we assume, as stated in the theorem, that $\eta \in \Omega^1(M)$ is of the form $\eta = \delta^M \eta_1 + \eta_2$ with $\eta_1 \in \Omega^1(M)$ and $\eta_2 \in \Omega^2(M)$ both satisfy $\nu \lrcorner \eta_i = 0$ on $\partial M$ and $d^M \eta_2 = \delta^M \eta_2 = 0$. The following remark explains that this is not really a restriction because of gauge invariance arguments.

\begin{remark} \label{rem:hodgedecomp}
Following \cite[Prop. 1(2)]{CSIS-21}, any $1$-form $\eta \in \Omega^1(M)$ can be written as $\eta = \tilde \eta + d^M f$ with a smooth function $f \in C^\infty(M)$ and $\tilde \eta \in \Omega^1(M)$ satisfying $\delta^M \tilde \eta = 0$ and $\nu \lrcorner \tilde \eta = 0$ on $\partial M$. Since $d^Mf \in \mathfrak{B}_M$, by property (iii) of Remark \ref{rem:ShikegawaNeumann}, $\eta$ and $\tilde \eta$ are gauge-equivalent. Moreover, by the Hodge-Morrey-Friedrichs Decomposition (see \cite[p. 3]{Sch:95}), $\tilde \eta$ can be decomposed into
$$ \tilde \eta = d^M\tilde f + \delta^M \omega + h $$
with $\tilde f \in C^\infty(M)$ vanishing on $\partial M$, $\omega \in \Omega^2(M)$ satisfying $\nu \lrcorner \omega = 0$ on $\partial M$, and
$h \in \Omega^1(M)$ satisfying $d^M h = \delta^M h = 0$. Since $\delta^M \tilde \eta=0$, we have $\Delta^M \tilde f = \delta^M d^M \tilde f = 0$, and since $\tilde f$ vanishes at the boundary $\partial M$, $\tilde f =0$. This leads to the decomposition in \cite[Prop 1(3)]{CSIS-21}
$$ \tilde \eta = \delta^M \omega + h, $$
where we have the additional property $\nu \lrcorner h = 0$ on $\partial M$, since $\tilde \eta$ and $\delta^M \omega$ both satisfy $\nu \lrcorner \tilde \eta = 0$ and $\nu \lrcorner(\delta^M \omega) = 0$ on $\partial M$. In conclusion, we can restrict to magnetic potentials of the form
$$ \eta = \delta^M \eta_1 + \eta_2 $$
with $\eta_1 \in \Omega^2(M)$ satisfying $\nu \lrcorner \eta_1 =0$ and $\eta_2 \in \Omega^1(M)$ satisfying $d^M \eta_2 = \delta^M \eta_2 =0$ and $\nu \lrcorner \eta_2 = 0$ on $\partial M$ for any spectral considerations of the magnetic Steklov operator $T^\eta$, by gauge-invariance.
\end{remark}

\begin{proof}[Proof of Theorem \ref{thm:colboissavo}]
Let $x_0 \in M$ and $\omega \in \mathfrak{L}_{\mathbb{Z}}$. Then the function $u(x) = e^{i \phi(x)}$ with
$$ \phi(x) = \int_{x_0}^x \omega $$
satisfies $|u(x)| = 1$ and 
$$ d^\eta u = i u (\omega + \eta_2 + \delta^M \eta_1). $$
By the orthogonality of the components of the Hodge decomposition, this implies
$$ \sigma_1^\eta(M) \le 
\frac{\Vert d^\eta u \Vert_{L^2(M)}^2}{\Vert u \Vert^2_{L^2(\partial M)}} = \frac{\Vert \eta_2 + \omega \Vert_{L^2(M)}^2+\Vert \delta^M \eta_1 \Vert^2_{L^2(M)}}{|\partial M|}. 
$$
Since $d^M \eta = d^M \delta^M \eta_1$, we have
$$ \lambda''_{1,1}(M) \le \frac{\Vert d^M \delta^M \eta_1 \Vert^2_{L^2(M)}}{\Vert \delta^M \eta_1 \Vert^2_{L^2(M)}} = \frac{\Vert d^M \eta \Vert^2_{L^2(M)}}{\Vert \delta^M \eta_1 \Vert^2_{L^2(M)}}, $$
with $\lambda_{1,1}''(M)$ defined in \eqref{eq:lambda11''}. Therefore, we obtain
$$ \sigma_1^\eta(M) \le \frac{1}{|\partial M|} \left( {\rm{dist}}(\eta_2, \mathfrak{L}_{\mathbb{Z}})^2 + \frac{\Vert d^M \eta \Vert^2_{L^2(M)}}{\lambda_{1,1}''(M)} \right). $$
\end{proof}

\section{Examples: Computation of the full magnetic Steklov spectrum} 
 
\subsection{Magnetic Steklov spectrum for the \texorpdfstring{$2$}{2}-dimensional Euclidean ball}

In this subsection, we derive the full spectrum, given in Example \ref{ex:2disk}, of the magnetic Steklov operator $T^{t \eta}$ with
$$ \eta = -y dx + x dy $$
on the Euclidean unit ball $\mathbb{B}^2 \subset \mathbb{R}^2$, centered at the origin.

\medskip

The Laplacian on $\mathbb{R}^2$, in polar coordinates $(r,\theta)$, is given by
$$ 
\Delta^{\mathbb{R}^2}=-\partial^2_{rr}-\frac{1}{r}\partial_r-\frac{1}{r^2}\partial_{\theta\theta}^2.
$$ 
The eigenvalues of the Laplacian on $\mathbb{S}^1$ are $k^2$ for integers $k\geq 0$. The eigenvalue $k=0$ is simple with constant eigenfunctions, and all other eigenvalues for $k \ge 1$ have multiplicity $2$ with eigenfunctions $e^{ik\theta}$ and $e^{-ik\theta}$. It is not difficult to check from the expression of $\Delta^{\mathbb{R}^2}$ that the functions $r^k e^{ik\theta}$ and $r^k e^{-ik\theta}$ for integers $k \ge 0$ are harmonic functions on $\mathbb{R}^2$. Next, we will find the $\eta$-harmonic extension $\hat f_k \in C^\infty(\mathbb{B}^2)$ of $f_k=e^{ik\theta} \in C^\infty(\mathbb{S}^1)$ (as well as of $\overline{f}_k$), by setting $\hat f_k:=Q(r) r^k f_k$ where $Q$ is a smooth function in $r$, to be determined. Using $|\eta|^2=r^2$, $\eta(\hat f_k)=ik\hat f_k$ and the fact that $\delta^M \eta = 0$ (since $\eta = X^\flat$ with a Killing vector field $X$, which is divergence free), we obtain
\begin{eqnarray*}
\Delta^{t\eta}\hat f_k&=&\Delta^{\mathbb{R}^2}\hat f_k-2it\eta(\hat f_k)+t^2r^2\hat f_k\\
&=&r^kf_k(\Delta^{\mathbb{R}^2}Q) +Q\cdot\Delta^{\mathbb{R}^2}(r^kf_k)-2g(d^{\mathbb{R}^2}Q,d^{\mathbb{R}^2}(r^kf_k))+2ktQr^kf_k +t^2r^{k+2}Qf_k\\
&=&r^k\left((-Q''-\frac{1}{r}Q')f_k-\frac{2k}{r}Q'f_k+2ktQf_k+t^2r^2Qf_k\right)\\
&=&\left(-Q''-\frac{(2k+1)}{r}Q'+(2kt+t^2r^2)Q\right)r^kf_k.
\end{eqnarray*}
Hence, $\hat f_k$ is $\eta$-harmonic if and only if $Q$ satisfies the differential equation 
$$Q''+\frac{(2k+1)}{r}Q'-(2kt+t^2r^2)Q=0,$$
with the initial condition $Q'(0)=0$. The solutions of such a differential equation satisfying  $Q(1)=1$ are given by 
$$Q(r)=\frac{e^\frac{t(1-r^2)}{2} L^{(-k)}_{-\frac{1}{2}}(t r^2)}{r^{2k}L^{(-k)}_{-\frac{1}{2}}(t)},$$
where $L_a^{(b)}$ is the Laguerre function with parameters $a,b$. The same computation  can be done for the function $\overline{f}_k$, and we obtain the corresponding differential equation, given by 
$$\widetilde Q''+\frac{(2k+1)}{r}\widetilde Q'-(-2kt+t^2r^2)\widetilde Q=0.$$
The corresponding solutions are then given by 
$$\widetilde Q(r)=\frac{e^\frac{-t(1-r^2)}{2} L^{(-k)}_{-\frac{1}{2}}(-t r^2)}{r^{2k}L^{(-k)}_{-\frac{1}{2}}(-t)}.$$
Now, applying the Steklov operator to the function $f_k=e^{ik\theta} \in C^\infty(\mathbb{S}^1)$ (resp. $e^{-ik\theta}$) yields
$$T^\eta (f_k)=-\nu \lrcorner d^\eta\hat f_k=\partial r\lrcorner (d(Qr^kf_k)+i\hat f_k\eta)=(Q'(1)+k) f_k,$$
which means that the corresponding eigenvalues are given by 
$$ Q'(1) + k = \frac{-(k+t+1)L_{-\frac{1}{2}}^{(-k)}(t)+(1+2k)L_{-\frac{3}{2}}^{(-k)}(t)}{L_{-\frac{1}{2}}^{(-k)}(t)}. $$
In the same way, we obtain the eigenvalues corresponding to the eigenfunction $\bar f_k=e^{-ik\theta}$ as follows:
$$ \widetilde Q'(1) + k = \frac{-(k-t+1)L_{-\frac{1}{2}}^{(-k)}(-t)+(1+2k)L_{-\frac{3}{2}}^{(-k)}(-t)}{L_{-\frac{1}{2}}^{(-k)}(-t)}. $$

\subsection{Magnetic Steklov spectrum for the \texorpdfstring{$4$}{4}-dimensional Euclidean ball} \label{subsec:ex4dimball}

We start by recalling some spectral facts about the magnetic Laplacian $\Delta^{t \eta}$ on the $3$-dimensional unit sphere $\mathbb{S}^3\subset\mathbb{C}^2$, equipped with the standard Riemannian metric $g$ of constant curvature, where $\eta \in \Omega^1(\mathbb{S}^3)$ is given by 
$$ \eta := -y_1 \partial_{x_1} + x_1 \partial_{y_1} - y_2 \partial_{x_2} + x_2 \partial_{y_2}. $$
Note that $\eta^\sharp$ is the Killing vector field corresponding to the Hopf fibration. We refer to \cite{EGHP:23} for more details. We have shown in \cite{EGHP:23} that the spectrum of the corresponding magnetic Laplacian $\Delta^{t\eta}$ on complex  smooth functions is given by
\begin{equation} \label{eq:specmagS3}
k(k+2)+2(2p-k)t+t^2, \quad k\in\NN\cup\{0\},\,\, p\in\{0,\ldots, k\}.
\end{equation}
with multiplicities $(k+1)^2$ and the corresponding eigenfunctions are $f_{p,k}:=u^pv^{k-p}$, where $u=az_1+bz_2$ and $v=b\overline{z}_1-a\overline{z}_2$ for $(a,b)\in \mathbb{C}^2\setminus\{(0,0)\}$. The functions $f_{p,k}$ are the restrictions of homogeneous harmonic polynomials on $\mathbb{C}^2$ of degree $k$ to the unit sphere  $\mathbb{S}^3$. A straightforward computation yields (see \cite[p. 30]{Hi:74} or \cite[Lemma III.7.1]{Pe:93})
\begin{equation}\label{eq:Y2phi} 
\eta(f_{p,k})=i(2p-k)f_{p,k}, 
\end{equation}
that we will use later in our computations. Moreover, instead of using the pairs $(p,k)$ with the conditions in \eqref{eq:specmagS3}, we use a more symmetric choice $(p_1,p_2)$, related to $(p,k)$ by $p_1=p$ and $p_2=k-p$. The spectrum of the corresponding magnetic Steklov operator $T^{t \eta}$ as functions of $t$ is illustrated in Figure \ref{fig:spectrum-comparison2} and given in the following theorem. 

\begin{figure}
     \begin{center}    \includegraphics[width=0.5\textwidth]{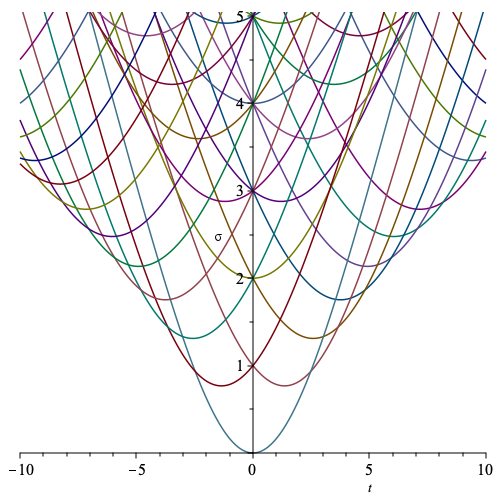}
     \end{center}
     \caption{Eigenvalues of the magnetic Steklov operator $T^{t \eta}$ on $\partial \mathbb{B}^4$ as functions in $t$.
     \label{fig:spectrum-comparison2}}
\end{figure}

\begin{theorem} Let $(\mathbb{B}^4,g)$ be the Euclidean ball equipped with the metric $g=dr^2\oplus r^2 d\theta^2$ and let $\eta:=X^\flat$ with the Killing vector field 
$$ X(x_1,y_1,x_2,y_2):=-y_1 \partial_{x_1} + x_1 \partial_{y_1} - y_2 \partial_{x_2} + x_2 \partial_{y_2} $$ 
in $\mathbb{R}^4$. The spectrum of the magnetic Steklov operator on smooth functions on $\mathbb{S}^3$ associated with the magnetic potential $t\eta$ ($t>0$) is given by
$$ {\rm{spec}}(T^{t \eta}) =  \left\{ \sigma(p_1,p_2) = \frac{F(p_1,p_2,t) - F(p_2,p_1,-t)}{G(p_1,p_2,t)-G(p_2,p_1,-t)}: (p_1,p_2) \in \left(\mathbb{N}\cup \{0\}\right)^2 \right\}$$
with
\begin{align*}
F(p_1,p_2,t) &:= e^{t/2} \sum_{j=0}^{p_1} (2j-p_1-p_2+t-2) (p_1+p_2-j)! \begin{pmatrix} p_1\\j \end{pmatrix} (-t)^j, \\
G(p_1,p_2,t) &:= e^{t/2} \sum_{j=0}^{p_1} (p_1+p_2-j)! \begin{pmatrix} p_1\\j \end{pmatrix} (-t)^j.
\end{align*}
In particular, the eigenvalue corresponding to $(p_1,p_2)=(0,0)$ is equal to 
$$ \sigma(0,0) = t{\rm coth}(t/2)-2, $$
which goes to $0$ as $t\to 0$.  
Therefore, for small $|t|$, the lowest magnetic Steklov eigenvalue is $t{\rm coth}(t/2)-2$.
\end{theorem} 

\begin{proof}
In the following, we will find the $\eta$-harmonic extensions of the complex functions $f_{p,k}=u^pv^{k-p}$ defined previously, where $u=az_1+bz_2$ and $v=b\overline{z}_1-a\overline{z}_2$ for $(a,b)\in \mathbb{C}^2\setminus\{(0,0)\}$. Recall that $f_{p,k}$ are the restrictions of homogeneous harmonic polynomials on $\mathbb{C}^2$ of degree $k$ to the unit sphere  $\mathbb{S}^3$ and are eigenfunctions of the Laplacian on $\mathbb{S}^3$ associated with the eigenvalues $k(k+2)$. Now, we search for such an extension to be of the form $\hat f_{p,k}:=Qf_{p,k}$, where as usual $Q$ is a smooth function in $r$, to be determined. Using $\eta(r)=0$, we write the equation for the magnetic Laplacian    
\begin{eqnarray*}
\Delta^{t\eta}\hat f_{p,k}&=&\Delta^{\mathbb{R}^4}\hat f_{p,k}-2it\eta(\hat f_{p,k})+t^2r^2\hat f_{p,k}\\
&\stackrel{\eqref{eq:Y2phi}}{=}& f_{p,k}(\Delta^{\mathbb{R}^4}Q)+Q(\Delta^{\mathbb{R}^4}f_{p,k})-2g(d^{\mathbb{R}^4}Q,d^{\mathbb{R}^4}f_{p,k})+2t(2p-k)Qf_{p,k} +t^2r^2Qf_{p,k}\\
&=&(-Q''-\frac{3}{r}Q')f_{p,k}-\frac{2k}{r}Q'f_{p,k}+2t(2p-k)Qf_{p,k}+t^2r^2Qf_{p,k}\\
&=&\left(-Q''-\frac{2k+3}{r}Q'+(2t(2p-k)+t^2r^2)Q\right)f_{p,k}.
\end{eqnarray*}
Here, we use also $\partial_r(f_{p,k})=\frac{k}{r}f_{p,k}$, which can be proven straightforwardly. Hence $\hat f_{p,k}$ is $\eta$-harmonic if and only if $Q$ solves the differential equation
\begin{equation}\label{eq:q''}
Q''+\frac{2k+3}{r}Q'-(2t(2p-k)+t^2r^2)Q=0.
\end{equation}
According to Wolfram Alpha, the general solution of this differential equation satisfying $Q'(0)=0$ is of the form 
$$Q(r)=\frac{e^{tr^2/2}A(r,t)-e^{-tr^2/2}B(r,t)}{r^{2k+2}t^{{\rm min}(p,k-p)}},$$ 
where $A(r,t)$ and $B(r,t)$ are polynomials in $r^2$ of degrees $2p$ and $2(k-p)$ respectively, with $A(0,t)=B(0,t)$ for any $t$. They are also polynomials in $t$ of degrees $p$ and $k-p$ respectively. In the following, we will determine $A$ and $B$. Since $e^{tr^2/2}$ and $e^{-tr^2/2}$ are linearly independent, the functions $\frac{e^{tr^2/2}}{r^{2k+2}}A(r,t)$ and $\frac{e^{-tr^2/2}}{r^{2k+2}}B(r,t)$ satisfy the same differential equation \eqref{eq:q''}. An easy computation shows that 
$$\left(\frac{e^{tr^2/2}}{r^{2k+2}}A(r,t)\right)'=e^{tr^2/2}\left(\frac{A'}{r^{2k+2}}+\frac{tA}{r^{2k+1}}-\frac{(2k+2)A}{r^{2k+3}}\right),$$
and 
\begin{align*}
\left(\frac{e^{tr^2/2}}{r^{2k+2}}A(r,t)\right)''=e^{tr^2/2}\Big(\frac{A''}{r^{2k+2}}+A'\left(-\frac{4k+4}{r^{2k+3}}+\frac{2t}{r^{2k+1}}\right)\\+A\left(\frac{t^2}{r^{2k}}+\frac{(2k+3)(2k+2)}{r^{2k+4}}-\frac{(4k+3)t}{r^{2k+2}}\right)\Big).
\end{align*}
Substituting these two derivatives into \eqref{eq:q''} yields the following differential equation for $A$: 
$$A''+(2tr-\frac{2k+1}{r})A'-4ptA=0.$$
Since $A$ is polynomial, we can write it as $A(r,t)=\sum_{j=0}^{p}a_j r^{2j}$, where $a_j$ depend only on $t$, for all $j$. Plugging the expression of $A$ into the above differential equation yields the recursion 
$$ka_1=-pt a_0,\,\, a_j=\frac{t(p-j+1)}{j(j-k-1)} a_{j-1}$$
for all $j=2,\ldots p$. Hence, for $k\neq 0$, we deduce that 
$$a_j=\frac{ (k-j)!\begin{pmatrix}p\\j\end{pmatrix}}{k!}(-t)^{j}a_0$$
for all $j=1,\ldots, p$. Now, one can easily check that for $k=0$, $A(r,t)=a_0$. Hence, $A$ becomes equal to 
$$A(r,t)=a_0 \sum_{j=0}^p \frac{ (k-j)! \begin{pmatrix}p\\j\end{pmatrix}}{ k!}(-tr^2)^{j},$$
for all $k$. The same can be done for $B$ and the corresponding differential equation is 
$$B''-(2tr+\frac{2k+1}{r})B'+4t(k-p)B=0.$$
Hence, since $A(0,t)=B(0,t)$, the function $B$ is equal to $$B(r,t)=a_0 \sum_{j=0}^{k-p} \frac{ (k-j)!\begin{pmatrix}k-p\\j\end{pmatrix}}{k!}(tr^2)^{j}.$$ 
Thus, we find after taking into account $Q(1)=1$, the following expression  
$$Q(r)=\frac{e^{tr^2/2}\sum_{j=0}^p  (k-j)!\begin{pmatrix}p\\j\end{pmatrix} (-tr^2)^{j}-e^{-tr^2/2}\sum_{j=0}^{k-p} (k-j)!\begin{pmatrix}k-p\\j\end{pmatrix}(tr^2)^{j}}{r^{2k+2}\left(e^{t/2}\sum_{j=0}^p (k-j)!\begin{pmatrix}p\\j\end{pmatrix}(-t)^{j}-e^{-t/2}\sum_{j=0}^{k-p}  (k-j)!\begin{pmatrix}k-p\\j\end{pmatrix} t^{j}\right)}.$$
To compute the eigenvalues, we apply the Steklov operator to $f_{p,k}|_{\mathbb{S}^1}$ to get
$$T^{[0],\eta} (f_{p,k}|_{\mathbb{S}^1})=-\nu \lrcorner d^\eta\hat f_k=\partial r\lrcorner (d(Qf_{p,k})+i\hat f_{p,k}\eta)=(Q'(1)+k) (f_{p,k}|_{\mathbb{S}^1}),$$
which means that the eigenvalues are given by $Q'(1)+k$. Computing the derivative of the above function $Q$ gives the spectrum, as required. 
\end{proof}

\section{Comparison of magnetic Steklov and magnetic Laplacian eigenvalues}
\label{sec:speccomparison}

We first derive a magnetic Pohozhaev identity in the general framework of complex differential forms, which could also be useful in other contexts. See \cite[Theorem 5.2]{GKLP:22} for a Pohozhaev identity for real differential forms without magnetic potential. Recall from \cite{EGHP:23} that for a magnetic potential (i.e., a smooth real $1$-form) $\eta$, we have the magnetic exterior differential $d^\eta:=d+i\eta\wedge$, magnetic codifferential $\delta^\eta:=\delta-i\eta\lrcorner$ and the magnetic Hodge Laplacian $\Delta^\eta:=d^\eta\delta^\eta+\delta^\eta d^\eta$.

\begin{theorem}[magnetic Pohozhaev Identity] \label{thm:magnetic pohozhaev}
Let $(M^m,g)$ be a compact Riemannian manifold with smooth boundary $\partial M$. Let $\eta\in\Omega^1(M)$ and $F$ be a real Lipschitz vector field on $M$. Then, for every $\omega\in \Omega ^p(M,\CC)$ satisfying $\delta^\eta d^\eta\omega=0$, the following holds true.
\begin{multline} \label{eq:magnetic pohozhaev}
2\Re \int_{\partial M}\langle \iota^*(F\lrcorner d^\eta\omega),\nu\lrcorner d^\eta\omega\rangle \dvpm- \int_{\partial M} \langle F,\nu \rangle |d^\eta \omega|^2 \dvpm\\
=2\Re \,i \int_M \langle F\lrcorner(d^M\eta \wedge \omega),d^\eta \omega\rangle\dvm +\int_M|d^\eta \omega|^2\, \dvg F\dvm +\int_M (\lie_F g)(d^\eta \omega,d^\eta \omega) \dvm.
\end{multline}
Here, $\lie_F$ denotes the Lie derivative with respect to $F$.
\end{theorem}

\begin{proof}
From the magnetic Green's formula (see \cite[Section 5.1]{EGHP:23}) and the divergence theorem, we have
\begin{align}
2&\Re \int_{\partial M}\langle \iota^*(F\lrcorner d^\eta\omega),\nu\lrcorner d^\eta\omega\rangle \dvpm- \int_{\partial M} \langle F,\nu \rangle |d^\eta \omega|^2 \dvpm \nonumber\\
&=2\Re \left(\int_M\langle F\lrcorner d^\eta\omega, \delta^\eta d^\eta \omega \rangle \dvm-\int_M \langle d^\eta(F\lrcorner d^\eta\omega), d^\eta\omega\rangle \dvm\right)+\int_M \dvg (|d^\eta \omega|^2 F) \dvm \nonumber\\
&=-2\Re \int_M \langle d^\eta(F\lrcorner d^\eta\omega), d^\eta\omega\rangle \dvm +\int_M |d^\eta \omega|^2 \dvg F + F(|d^\eta\omega|^2)\dvm. \label{eq:eq1 in pohoz proof}
\end{align}
Now,
\begin{align*}
d^\eta(F\lrcorner d^\eta \omega)&=d^M(F\lrcorner d^\eta \omega)+i\eta\wedge(F\lrcorner d^\eta \omega)\\
&=\lie_Fd^\eta \omega-F\lrcorner(d^Md^\eta \omega)+i\eta\wedge(F\lrcorner d^\eta \omega)\\
&=\lie_Fd^\eta \omega - F\lrcorner(d^M(d^M\omega+i\eta\wedge \omega))+i \eta \wedge (F\lrcorner d^\eta \omega)\\
&=\lie_F d^\eta \omega - iF\lrcorner(d^M\eta\wedge \omega)+iF\lrcorner (\eta\wedge d^M\omega)+i\eta \wedge (F\lrcorner d^\eta \omega)\\
&=\lie_F d^\eta \omega - iF\lrcorner(d^M\eta\wedge \omega)+iF\lrcorner (\eta\wedge d^\eta\omega)+i\eta \wedge (F\lrcorner d^\eta \omega)\\
&=\lie_F d^\eta \omega - iF\lrcorner(d^M\eta\wedge \omega)+i\langle F, \eta \rangle d^\eta \omega.
\end{align*}
Hence,
\begin{align}
2\Re \langle d^\eta(F\lrcorner d^\eta\omega),d^\eta \omega\rangle &=2\Re\left(\langle \lie_F d^\eta \omega, d^\eta \omega\rangle -i \langle F\lrcorner (d^M\eta\wedge \omega), d^\eta \omega \rangle+i \langle F,\eta \rangle |d^\eta \omega|^2\right)\notag\\
&=F(|d^\eta \omega|^2)-(\lie_F g) (d^\eta\omega,d^\eta \omega) - 2\Re i \langle F\lrcorner (d^M\eta\wedge \omega), d^\eta \omega \rangle, \label{eq:eq2 in pohoz proof}
\end{align}
since the Lie derivative of the metric tensor on complex $p$-forms expands as
\begin{equation*}
(\lie_F g)(\omega_1,\omega_2)=F (g(\omega_1,\omega_2))- g(\lie_F \omega_1, \omega_2)- g(\omega_1, \lie_F \omega_2).
\end{equation*}
Finally, we substitute the expression from \eqref{eq:eq2 in pohoz proof} into equation \eqref{eq:eq1 in pohoz proof} to get the desired identity \eqref{eq:magnetic pohozhaev}.
\end{proof}

In what follows, we denote by $\bar h = {\rm roll}(M)$ the rolling radius of $M$, that is the distance between $\partial M$ and its cut-locus. For any fixed $h \in (0,\bar h)$, we denote by $\Sigma_h$ the parallel hypersurface at a distance $h$ from $\partial M$ and by $M_h:=\{x\in M: {\rm dist}(x,\partial M)<h\}$ the tubular neighborhood of $\partial M$. We let $\alpha$ and $\beta$ be the infimum and supremum of the sectional curvature of $M_h$, and denote by $\kappa_-$ and $\kappa_+$ the infimum and supremum of the principal curvatures of $\partial M$. Let the function $\gamma: M_h\to \RR$ be given by
$$
\gamma(x):=\half {\rm dist}(x,\Sigma_h)^2.
$$
Note that $\grad \gamma$ is normal to the parallel hypersurfaces in $M_h$. 

In order to state the comparison result, we first apply the magnetic Pohozhaev inequality stated in Theorem \ref{thm:magnetic pohozhaev} (see \cite[Lemma 21]{CGH:20} for the case without magnetic potential).

\begin{lemma} \label{lem:lem21 of CGH20}  Let $\eta\in \Omega^1(M)$ and $f$ be a $\eta$-harmonic complex valued function. Then
\begin{eqnarray*}
 \int_\pam | d^{\iota^*\eta} f|^2-|\nu \lrcorner d^\eta f|^2\,\dvpm &=&-\frac{2}{h}\int_{M_h} {\rm Im}(f\langle\grad \gamma\lrcorner d^M\eta , d^\eta f\rangle) \dvm\\&&-\frac{1}{h}\int_{M_h} ((\Delta \gamma) \,|d^\eta f|^2+2\hess_\gamma( d^\eta f,  d^\eta f))\,\dvm,
 \end{eqnarray*}
where $d^{\iota^*\eta}:=d^{\partial M}+i(\iota^*\eta)\wedge.$
\end{lemma}

\begin{proof}
Applying the magnetic Pohozhaev Identity \eqref{eq:magnetic pohozhaev} with $\omega=f$ and the Lipschitz vector field $F$ given by 
\begin{equation*}
F=\left\{
\begin{matrix}
 \grad\gamma&  \textrm{on}\,\, M_h\\\\
0\,\,  &\textrm{on}\,\, M\setminus M_h
\end{matrix}\right.
\end{equation*}
yields, after using $F|_\pam=-h\nu$ and $F|_{\Sigma_h}=0$, the following
\begin{multline*}
-2h \int_\pam |\nu\lrcorner d^\eta f|^2 \dvpm +h\int_\pam |d^\eta f|^2\dvpm \\=2 \Re i \int_{M_h} f\langle\grad \gamma\lrcorner d^M\eta , d^\eta f\rangle \dvm - \int_{M_h}  (\Delta\gamma)|d^\eta f|^2 \dvm + \int_{M_h} (\lie_{\grad \gamma}g)(d^\eta f,d^\eta f) \dvm.
\end{multline*}
Now, it is not difficult to check that $(\lie_{\grad \gamma} g)(d^\eta f, d^\eta f)=-2 \hess_\gamma ( d^\eta f,  d^\eta f)$. Also by writing $|d^\eta f|^2=| \iota^*(d^\eta f)|^2 + |\nu \lrcorner d^\eta f|^2$ at any point in $\partial M$ and using that $\iota^*(d^\eta f)=d^{\iota^*\eta} f$  which can be proven straightforwardly, we deduce the equality.
\end{proof}

Henceforth, we assume that $h\leq \tilde h$, where $\tilde h$ is the constant defined as in \cite[Theorem 12]{CGH:20} that depends on the supremum of principal curvatures of $\pam$ and the supremum of sectional curvature of $M_{\bar h}$.
\begin{lemma} \label{lem:lem22 of CGH20} (cf. \cite[Lemma 22]{CGH:20} Let $\{\kappa_i\}_{i=1}^{m-1}$ denote the principal curvatures of $\pam$. For $h\in (0,\tilde h)$ and $t \in (0,h)$, the eigenvalues of $\hess_\gamma$ at $x\in \Sigma_t$ are given by
$$
\rho_1(x)=(h-t)\kappa_1(x)\leq \dots\leq \rho_{m-1}(x)=(h-t)\kappa_{m-1}(x)\leq \rho_m(x)=1.
$$
\end{lemma}

For the following lemma, note that \cite{CGH:20} uses the negative Laplacian, whereas we use the positive Laplacian. 

\begin{lemma} \label{lem:lem23 of CGH20} (cf. \cite[Lemma 23]{CGH:20}) Let $f$ be a smooth complex valued function, then, on $M_h$, we have
$$
-\left(1-\sum_{j=1}^{m-1}\rho_j\right)|d^\eta f|^2\leq -(\Delta\gamma)\,|d^\eta f|^2-2\hess_\gamma ( d^\eta f,  d^\eta f)\leq\left(1+\left(\sum_{j=2}^{m-1}\rho_j\right)-\rho_1\right)|d^\eta f|^2.
$$
\end{lemma}

\begin{lemma} \label{lem:lem24 of CGH20} (cf. \cite[Lemma 24]{CGH:20}) Let $\alpha, \beta, \kappa_-, \kappa_+ \in \RR$ be as described before. There exist constants $\bar A_h= \bar A_h(m,\alpha,\beta, \kappa_-,\kappa_+)$ and $\bar B_h=\bar B_h (m,\alpha,\kappa_-)$ such that
$$
-(1+\bar B_h)\int_M |d^\eta f|^2 \dvm\leq \int_{M_h} -(\Delta \gamma)\,|d^\eta f|^2-2\hess_\gamma ( d^\eta f,  d^\eta f)\,\dvm \leq (1+\bar A_h) \int_M|d^\eta f|^2\dvm.
$$
\end{lemma}
The proof of Lemma \ref{lem:lem24 of CGH20} is almost the same as that of \cite[Lemma 24]{CGH:20}, with some straightforward modifications.

\begin{proposition} \label{prop:thm27 of CGH20} (cf. \cite[Theorem 27]{CGH:20} for the non-magnetic analogue) Let $\alpha, \beta, \kappa_-, \kappa_+ \in \RR$ and $\tilde h$ be as described before. Let $h \in (0,\tilde h)$. Choose $A_h, B_h \in \RR$ satisfying
$$
A_h\geq \frac{1}{h}(1+\bar A_h) \quad \text{and} \quad B_h\geq \frac{1}{2h}(1+\bar B_h).
$$
 Assume that $\eta\notin \mathfrak{B}_M$, i.e. $\eta$ is not gauged away. Then for any $\eta$-harmonic function $f\in C^\infty (M,\CC)$ normalized by \hbox{$\int_\pam |f|^2\dvpm=1$}, we have
\begin{gather*}
\int_\pam | d^{\iota^*\eta} f|^2\dvpm \leq\int_\pam |\nu\lrcorner d^\eta f|^2 \dvpm +\widetilde A_h\left(\int_\pam |\nu\lrcorner d^\eta f|^2 \dvpm\right)^\half,\\
\left(\int_\pam |\nu\lrcorner d^\eta f|^2 \dvpm\right)^\half \leq \widetilde B_h+\sqrt{\widetilde B_h^2+\int_\pam | d^{\iota^*\eta} f|^2\dvpm}.
\end{gather*}
where $\widetilde A_h:=A_h+\frac{2||\grad\gamma\lrcorner d^M\eta||_\infty}{h\sqrt{\sigma_1^\eta(M)}} \left (\int_0^R\Theta(r) dr\right)^\frac{1}{2}$ and $\widetilde B_h:=B_h+\frac{||\grad\gamma\lrcorner d^M\eta||_\infty}{h\sqrt{\sigma_1^\eta(M)}} \left (\int_0^R\Theta(r) dr\right)^\frac{1}{2}$ with $\Theta(r)=(s'_K(r)-\kappa_{-}s_K(r))^{m-1}$ where $s_K(r)$ is defined in \eqref{eq:skh} with $K$ a global lower sectional curvature bound of $M$ and $R:= \max\{{\rm{dist}}(x,\partial M): x \in M \}$.
\end{proposition}
\begin{proof}
Let $\eta_0 = \iota^* \eta \in \Omega^1(\partial M)$.
By Lemma \ref{lem:lem21 of CGH20} and Lemma \ref{lem:lem24 of CGH20}, we have
\begin{eqnarray}\label{eq:compad}
\int_\pam (| d^{\eta_0} f|^2-|\nu \lrcorner d^\eta f|^2)\,\dvpm &=&-\frac{2}{h}\int_{M_h} {\rm Im}(f\langle\grad \gamma\lrcorner d^M\eta , d^\eta f\rangle) \dvm\nonumber\\&&-\frac{1}{h}\int_{M_h} ((\Delta \gamma) \,|d^\eta f|^2+2\hess_\gamma( d^\eta f,  d^\eta f))\,\dvm,\nonumber\\
&\leq& \frac{2}{h}||\grad\gamma\lrcorner d^M\eta||_{M_h,\infty}\left(\int_M |f|^2\dvm\right)^\frac{1}{2} \left(\int_M|d^\eta f|^2\dvm\right)^\frac{1}{2}\nonumber\\&&+ A_h\int_M |d^\eta f|^2\dvm
\end{eqnarray}
where we use the Cauchy-Schwarz inequality in the last step. 
Since $M$ has curvature bounds $(K,\kappa_{-})$ and since $f$ is $\eta$-harmonic, Proposition \ref{prop:L2estimate} gives that
\begin{equation}\label{eq:l2normboun}
\int_{M}|f|^2 {\rm dvol}_{M}\leq\left (\int_0^R\Theta(r) dr\right)\int_{\partial M}|f|^2 {\rm dvol}_{\partial M}=\left (\int_0^R\Theta(r) dr\right),
\end{equation}
since $\int_\pam |f|^2\dvpm=1$.  
On the other hand, as $\eta\notin \mathfrak{B}_M$, one has by Theorem  \ref{thm:shikegawa} that $\sigma_1^\eta(M)>0$. 
Therefore by the min-max principle in \eqref{eq:characsigma}
\begin{equation}\label{eq:sigma1f}
 \sigma_1^\eta(M) = \inf_{h \in C^\infty(M,\mathbb{C})} \frac{\int_M |d^\eta h|^2 {\rm dvol}_M}{\int_{\partial M} | h|^2 {\rm dvol}_{\partial M} }\leq \int_M |d^\eta f|^2 {\rm dvol}_M.
 \end{equation}
Substituting Inequalities \eqref{eq:l2normboun} and \eqref{eq:sigma1f} into \eqref{eq:compad}, we get 
\begin{eqnarray*}
\int_\pam (| d^{\iota^*\eta} f|^2-|\nu \lrcorner d^\eta f|^2)\,\dvpm 
&\leq& \left(\frac{2||\grad\gamma\lrcorner d^M\eta||_{M_h,\infty}}{h\sqrt{\sigma_1^\eta(M)}} \left (\int_0^R\Theta(r) dr\right)^\frac{1}{2}+ A_h\right)\int_M |d^\eta f|^2\dvm\\
&=& -\left(\frac{2||\grad\gamma\lrcorner d^M\eta||_{M_h,\infty}}{h\sqrt{\sigma_1^\eta(M)}} \left (\int_0^R\Theta(r) dr\right)^\frac{1}{2}+ A_h\right)\int_{\partial M}\langle f, \nu\lrcorner d^\eta f \rangle \dvpm\\
&\leq & \left(\frac{2||\grad\gamma\lrcorner d^M\eta||_{M_h,\infty}}{h\sqrt{\sigma_1^\eta(M)}} \left (\int_0^R\Theta(r) dr\right)^\frac{1}{2}+ A_h\right)\left(\int_\pam |\nu\lrcorner d^\eta f|^2 \dvpm\right)^\half,
\end{eqnarray*}
where we use the Cauchy-Schwarz inequality in the last step. In the same way, we can show
\begin{align*}
\int_\pam (| d^{\iota^*\eta} f|^2-|\nu \lrcorner d^\eta f|^2)\,\dvpm\geq -2\widetilde B_h \left(\int_\pam |\nu\lrcorner d^\eta f|^2 \dvpm\right)^\half,
\end{align*}
which can be solved by using a quadratic equation in $\left(\int_\pam |\nu\lrcorner d^\eta f|^2 \dvpm\right)^\half$, to get the second inequality in this lemma.
\end{proof}
\par \noindent
Finally, we state the following comparison theorem:
\begin{theorem} \label{thm:magnetic steklov-laplacian comparison}
Let $(M^m,g)$ be a compact Riemannian manifold with boundary $\partial M$ and $\eta \in \Omega^1(M)$ be a magnetic potential on $M$, such that $\eta\notin \mathfrak{B}_M$. Let $\eta_0 = \iota^* \eta$. Let $0<h<\tilde h$ be given. Then there exist constants $\widetilde A_h>0$ and $\widetilde B_h>0$ (given in Proposition \ref{prop:thm27 of CGH20}) depending on the geometry of $M$, the infinity norm of the magnetic field $d^M \eta$ and the first magnetic Steklov eigenvalue, such that 
\begin{align*}
\lambda^{\eta_0}_k(\partial M)&\leq \sigma^{\eta}_k(M)^2+\widetilde A_h \sigma^{\eta}_k(M) \quad \text{and}\\ 
\sigma^{\eta}_k(M) &\leq \widetilde B_h + \sqrt{\widetilde B_h^2+\lambda^{\eta_0}_k(\partial M)} \quad \forall \, k\in \NN.
\end{align*}
Therefore, we get $\Big|\sigma^{\eta}_k(M)-\sqrt{\lambda^{\eta_0}_k(\partial M)}\Big|<C_h:=\max\{\widetilde A_h,2\widetilde B_h\}$ for all $k\in \NN$.
\end{theorem}

\begin{proof} Let $\{f_j\}_{j\in \NN}$ be an orthonormal eigenbasis for $L^2(\pam, \CC)$ corresponding to the magnetic Steklov eigenvalues $\{\sigma^\eta_j(M)\}_{j\in \NN}$. We denote by  $\hat f_j$ the corresponding $\eta$-harmonic extension of each $f_j$.  We consider the complex vector space $E_k={\rm span}(f_j)_{j=1,\ldots,k}\in H^1(M,\mathbb{C})$. It is clear that each element $\phi$ in $E_k$ has an $\eta$-harmonic extension $\hat \phi$ in ${\rm span}(\hat f_j)_{j=1,\ldots,k}$. Using the variational characterization for the magnetic Laplacian eigenvalues of $\pam$, we have
\begingroup
\allowdisplaybreaks
\begin{align*}
\lambda^{\eta_0}_k (\pam) &= \min_{\begin{subarray}{c}V\subset H^1(\pam, \CC)\\ \dim_{\CC} V=k \end{subarray}}\max_{\begin{subarray}{c} \phi\in V\\ \int_\pam |\phi|^2 \dvpm=1 \end{subarray}} \int_\pam |d^{\eta_0} \phi|^2\dvpm\\
&\leq \max_{\begin{subarray}{c} \phi\in E_k\\ \int_\pam |\phi|^2 \dvpm=1 \end{subarray}} \int_\pam |d^{\eta_0} \phi|^2\dvpm\\
&\leq  \max_{\begin{subarray}{c} \phi\in E_k\\ \int_\pam |\phi|^2 \dvpm=1 \end{subarray}} \int_\pam |\nu\lrcorner d^\eta \hat\phi|^2 \dvpm +\widetilde A_h \left(\int_\pam |\nu\lrcorner d^\eta \hat\phi|^2 \dvpm\right)^\half \tag{by Proposition \ref{prop:thm27 of CGH20}}\\
&=\max_{\begin{subarray}{c} (c_1,\dots,c_k)\in \CC^k \\ |c_1|^2+\dots+|c_k|^2=1 \end{subarray}} \int_\pam |\nu\lrcorner d^{\eta_0} \big(\sum_{j=1}^k c_j f_j\big)|^2 \dvpm +\widetilde A_h \left(\int_\pam |\nu\lrcorner d^{\eta_0} \big(\sum_{j=1}^k c_j f_j\big)|^2 \dvpm\right)^\half\\
&=\max_{\begin{subarray}{c} (c_1,\dots,c_k)\in \CC^k \\ |c_1|^2+\dots+|c_k|^2=1 \end{subarray}}\left(\sum_{j=1}^k |c_j|^2(\sigma_j^\eta(M))^2+\widetilde A_h\left(\sum_{j=1}^k |c_j|^2(\sigma_j^\eta(M))^2\right)^\half\right)\\
&=(\sigma^\eta_k(M))^2+\widetilde A_h \sigma^\eta_k(M).
\end{align*}
\endgroup
For the second inequality of Theorem \ref{thm:magnetic steklov-laplacian comparison}, we proceed similarly. Let $\{f_k\}_{k\in\NN}$ be an orthonormal eigenbasis for $L^2(\pam, \CC)$ corresponding to the Laplacian eigenvalues $\{\lambda^{\eta_0}_k(\partial M)\}_{k\in\NN}$ and let $\{\hat f_k\}_{k\in\NN}$ denote the respective $\eta$-harmonic extensions to $M$. Let $E_k:={\rm span}(\hat f_j)_{j=1,\ldots,k}$. Clearly, any element in $E_k$ is $\eta$-harmonic. From the variational characterization for the $\eta$-Steklov eigenvalues, we have
\begingroup
\allowdisplaybreaks
\begin{align*}
\sigma^\eta_k(M)&=\min_{\begin{subarray}{c}V\subset H^1(M, \CC)
\\ \dim V=k \end{subarray}}\max_{\begin{subarray}{c} \phi\in V\\ \int_{\partial M}|\phi|^2 \dvpm=1 \end{subarray}} \int_M |d^{\eta}\phi|^2\dvm\\
&\leq \max_{\begin{subarray}{c} \phi\in E_k\\ \int_\pam |\phi|^2 \dvpm=1 \end{subarray}} \int_M |d^\eta \phi|^2\dvm\\
&= -\max_{\begin{subarray}{c} \phi\in E_k\\ \int_\pam |\phi|^2 \dvpm=1 \end{subarray}} \int_\pam \langle \phi,\nu\lrcorner d^\eta \phi\rangle\dvpm \tag{we use $\phi$ is $\eta$-harmonic}\\
&\leq \max_{\begin{subarray}{c} \phi\in E_k\\ \int_\pam |\phi|^2 \dvpm=1 \end{subarray}} \left(\int_\pam |\nu\lrcorner d^\eta\phi|^2\dvpm\right)^\half\\
&\leq \widetilde B_h+\left(\widetilde B_h^2+\max_{\begin{subarray}{c} \phi\in E_k\\ \int_\pam |\phi|^2 \dvpm=1 \end{subarray}}\int_\pam |d^{\eta} \phi|^2\dvpm\right)^\half\tag{we use Proposition \ref{prop:thm27 of CGH20}}\\
&=\widetilde B_h+\left(\widetilde B_h^2+\max_{\begin{subarray}{c} (c_1,\dots,c_k)\in \CC^k \\ |c_1|^2+\dots+|c_k|^2=1 \end{subarray}}\int_\pam |d^{\eta_0} \big(\sum_{j=1}^k c_j  f_j\big)|^2\dvpm\right)^\half\\
&=\widetilde B_h+\left(\widetilde B_h^2+\max_{\begin{subarray}{c} (c_1,\dots,c_k)\in \CC^k \\ |c_1|^2+\dots+|c_k|^2=1 \end{subarray}}\sum_{j=1}^k |c_j|^2 \lambda^{\eta_0}_j(\pam)\dvpm\right)^\half\\
&=\widetilde B_h+\sqrt{\widetilde B_h^2+\lambda^{\eta_0}_k(\pam)}.
\end{align*}
\endgroup
The proof of $|\sigma^\eta_k(M)-\sqrt{\lambda^{\eta_0}_k(\pam)}|\leq \max \{\widetilde A_h,2\widetilde B_h\}$ is identical to that of the proof in \cite[Section 4]{CGH:20}.
\end{proof}

\begin{remark}
    A careful analysis of the constants involved shows that, in the case $\grad \gamma \lrcorner d^M \eta = 0$ on $M_h$, the constant $C_h$ only depends on the local geometry of $M_h$ and is independent of the (strength of the) magnetic field $d^M \eta$.~Moreover, it is possible to obtain explicit expressions for the constants $A_h$ and $B_h$ in Lemma \ref{lem:lem24 of CGH20} in terms of the width $h$ of the tubular neighborhood of $\partial M$, the bounds on the sectional curvature of $M_h$ and the bounds on the principal curvatures of $\pam$, as done in \cite{CGH:20}.
\end{remark}

In order to get the independence of the constants $\widetilde A_h$ and $\widetilde B_h$ in Theorem \ref{thm:magnetic steklov-laplacian comparison} from the first magnetic Steklov operator, one can use the lower bound established in Theorem \ref{thm:magnjammes} in terms of the Cheeger constants. Also, it is possible to obtain some other lower bounds for the first eigenvalue of the magnetic Steklov operator in terms of the first magnetic Laplace eigenvalue of the boundary $\partial M$, as done in \cite[Thm. 8]{ES:97} and \cite[Thm. 1]{Xi:97}. For this, we recall the magnetic Reilly formula that was established in \cite{HK:18}: For any smooth function $f\in C^\infty(M,\CC)$, $\eta\in \Omega^1(M)$ and $\eta_0 =\iota^* \eta \in \Omega^1(\partial M)$, we have 
\begin{eqnarray*}
\int_M|{\rm Hess}^\eta f+\frac{1}{m}(\Delta^\eta f) g|^2 {\rm dvol}_M&=&\frac{m-1}{m}\int_M|\Delta^\eta f|^2{\rm dvol}_M-\int_M {\rm Ric}_M(d^\eta f,d^\eta f) {\rm dvol}_M\\&&+\int_M {\rm Im} \left(d^M\eta(d^\eta f,\overline{d^\eta f})\right) {\rm dvol}_M+\int_M|d^M\eta|^2 |f|^2{\rm dvol}_M\\
&&-(m-1)\int_{\partial M} H|\langle  d^\eta f,\nu\rangle|^2 {\rm dvol}_{\partial M} -2\int_{\partial M}\Re(\langle \nu,d^\eta f\rangle\Delta^{\eta_0} f){\rm dvol}_{\partial M}\\&&-\int_{\partial M} {\bf II}(d^{\eta_0} f,d^{\eta_0} f){\rm dvol}_{\partial M},
\end{eqnarray*}
where $\Delta^{\eta_0}$ is the magnetic Laplacian on $\partial M$ associated to the magnetic differential $d^{\eta_0}=d^{\partial M}+i\eta_0\wedge$. Using this formula, we get the following theorem.
\begin{theorem}\label{thm:lowbound}
Let $(M^m,g)$ be a compact Riemannian manifold with smooth boundary $\partial M$.  Let $\eta\in \Omega^1(M)$ be the magnetic potential such that $\eta_0=\iota^*\eta\notin \mathfrak{B}_{\partial M}$. Assume that ${\rm Ric}_M\geq ||d^M\eta||_\infty$ and ${\bf II}\geq \alpha>0$. We have the following estimate for the first eigenvalue of the magnetic Steklov eigenvalue 
$$\sigma_1^\eta(M)\geq \frac{\alpha}{2}-\frac{||d^M\eta||_\infty^2\left(\int_0^R\Theta(r)dr\right)  }{2\lambda_1^{\eta_0}(\partial M)},$$
with $\Theta(r)=({\rm cos}(\sqrt{K}r)-\frac{\alpha}{\sqrt{K}}{\rm sin}(\sqrt{K}r))^{m-1} $, where $K=\frac{||d^M\eta||_\infty}{m-1}$.
\end{theorem}
\begin{proof}
Let $f$ be any $\eta$-harmonic function. Then applying the above magnetic Reilly formula to $f$ yields
\begin{eqnarray*}
\int_M|{\rm Hess}^\eta f|^2 {\rm dvol}_M+\int_M {\rm Ric}_M(d^\eta f,d^\eta f) {\rm dvol}_M-\int_M {\rm Im} \left(d^M\eta(d^\eta f,\overline{d^\eta f})\right) {\rm dvol}_M\\=\int_M|d^M\eta|^2 |f|^2{\rm dvol}_M
-(m-1)\int_{\partial M} H|\langle  d^\eta f,\nu\rangle|^2 {\rm dvol}_{\partial M} -2\int_{\partial M}\Re(\langle \nu,d^\eta f\rangle\Delta^{\eta_0} f){\rm dvol}_{\partial M}\\-\int_{\partial M} {\bf II}(d^{\eta_0} f,d^{\eta_0} f){\rm dvol}_{\partial M}.
\end{eqnarray*}
Now, using that pointwise we have that $ {\rm Im} \left(d^M\eta(d^\eta f,\overline{d^\eta f})\right)\leq ||d^M\eta||_\infty |d^\eta f|^2$, we deduce that the l.h.s. of the above inequality is non-negative and, therefore, after using ${\rm II}\geq \alpha$, we obtain
\begin{eqnarray}\label{eq: reillyetahar}
0&\leq & ||d^M\eta||_\infty^2 \int_M |f|^2{\rm dvol}_M-(m-1)\int_{\partial M} H|\langle  d^\eta f,\nu\rangle|^2 {\rm dvol}_{\partial M} -2\int_{\partial M}\Re(\langle \nu,d^\eta f\rangle\Delta^{\eta_0} f){\rm dvol}_{\partial M}\nonumber\\&&-\alpha \int_{\partial M} |d^{\eta_0}f|^2 {\rm dvol}_{\partial M},\nonumber\\
&\leq &||d^M\eta||_\infty^2 \left(\int_0^R\Theta(r)dr\right)  \left(\int_{\partial M} |f|^2{\rm dvol}_{\partial M}\right)-(m-1)\int_{\partial M} H|\langle  d^\eta f,\nu\rangle|^2 {\rm dvol}_{\partial M}\nonumber\\&& -2\int_{\partial M}\Re(\langle \nu,d^\eta f\rangle\Delta^{\eta_0} f){\rm dvol}_{\partial M}-\alpha \int_{\partial M} |d^{\eta_0}f|^2 {\rm dvol}_{\partial M}.\nonumber\\
\end{eqnarray}
In the last inequality, we use Proposition \ref{prop:L2estimate}, since $M$ has curvature bounds $(||d^M\eta||_\infty, \alpha)$, and thus $\Theta(r)=({\rm cos}(\sqrt{K}r)-\frac{\alpha}{\sqrt{K}}{\rm sin}(\sqrt{K}r))^{m-1} $, where $K=\frac{||d^M\eta||_\infty}{m-1}$. Inequality \eqref{eq: reillyetahar} is true for any $\eta$-harmonic function $f$. In the particular case when $f$ is an eigenfunction of the magnetic Steklov operator associated with the first magnetic Steklov eigenvalue $\sigma_1^\eta(M)$ and the fact that $H\geq \alpha>0$, we deduce from \eqref{eq: reillyetahar} that 
\begin{eqnarray*}
0&\leq &||d^M\eta||_\infty^2 \left(\int_0^R\Theta(r)dr\right)  \left(\int_{\partial M} |f|^2{\rm dvol}_{\partial M}\right)
 +2\sigma_1^\eta(M)\int_{\partial M}\Re(\bar f\Delta^{\eta_0} f){\rm dvol}_{\partial M}\nonumber-\alpha \int_{\partial M} |d^{\eta_0}f|^2 {\rm dvol}_{\partial M}\\
 &=&||d^M\eta||_\infty^2 \left(\int_0^R\Theta(r)dr\right)  \left(\int_{\partial M} |f|^2{\rm dvol}_{\partial M}\right)
 +(2\sigma_1^\eta(M)-\alpha) \int_{\partial M} |d^{\eta_0}f|^2 {\rm dvol}_{\partial M}.
 \end{eqnarray*}
 Since $\eta_0\notin \mathfrak{B}_{\partial M}$ by assumption, we have by the min-max principle that 
$$\int_{\partial M} |f|^2{\rm dvol}_{\partial M}\leq \frac{1}{\lambda_1^{\eta_0}(\partial M)}\int_{\partial M} |d^{\eta_0}f|^2 {\rm dvol}_{\partial M}.$$ 
Therefore, we deduce that 
$$\sigma_1^\eta(M)\geq \frac{\alpha}{2}-\frac{||d^M\eta||_\infty^2\left(\int_0^R\Theta(r)dr\right)  }{2\lambda_1^{\eta_0}(\partial M)}.$$
\end{proof}

It would be interesting to investigate whether the lower bound in Theorem \ref{thm:lowbound} is positive, or which additional assumptions would be required in order for this lower bound to be positive.\\

{\bf Acknowledgment:} We would like to thank Asma Hassannezhad for many stimulating discussions. We are also grateful to Bernard Helffer and Thomas Zaslavsky for their helpful comments. 

This research was carried out during Georges Habib's 3-month visit to Durham University in 2024, which was supported by the Atiyah Lebanon-UK Fellowship AF-2023-01 of the LMS (London Mathematical Society), a Visiting Fellowship from the ICMS (International Centre for Mathematical Sciences) and CAMS (Center for Advanced Mathematical Sciences) at the American University of Beirut.
Part of this project was carried out when the authors met in Bristol. The authors acknowledge the EPSRC grant EP/T030577/1 which made this visit possible.\\


\begin{thebibliography}{9}
\bibitem{BSS:13} A. S. Bandeira, A. Singer and Amit and D. A. Spielman, \emph{A {C}heeger inequality for the graph connection {L}aplacian}, SIAM J. Matrix Anal. Appl. \textbf{34} (2013), no. 4, 1611--1630.
  \bibitem{Bu:82} P. Buser,  \emph{A note on the isoperimetric constant}, Ann. Sci. \'{E}cole Norm. Sup. \textbf{15} (1982), no. 2, 
213--230.
\bibitem{CS:24} M. Cekic and A. Siffert, \emph{Spectral asymptotics of the magnetic Dirichlet-to-Neumann map on surfaces}, $\mathtt{arXiv:2410.08591}$ (October 2024).
\bibitem{Ch:70} J. Cheeger, \emph{A lower bound for the smallest eigenvalue of the {L}aplacian}, in \emph{Problems in analysis ({S}ympos. in honor of {S}alomon {B}ochner, {P}rinceton {U}niv., {P}rinceton, {N}.{J}., 1969)}, Princeton Univ. Press, Princeton, NJ, 1970, pp. 195--199.
\bibitem{CSIS-21} B. Colbois, A. El Soufi, S. Ilias and A. Savo, \emph{Eigenvalue upper bounds for the magnetic Schr\"odinger operator}, Comm. Anal. Geom. \textbf{30} (2022), 779--814.
\bibitem{CGH:20} B. Colbois, A. Girouard and A. Hassannezhad, \emph{The Steklov and Laplacian spectra of Riemannian manifolds with boundary}, J. Funct. Anal. \textbf{278} (2020), no. 6, article no. 108409.
\bibitem{CLPS:23} B. Colbois, C. L\'{e}na, L. Provenzano and A. Savo, \emph{Geometric bounds for the magnetic {N}eumann eigenvalues in the plane}, J. Math. Pures Appl. (9) \textbf{179} (2023), 454--497.
\bibitem{CPS:22} B. Colbois, L. Provenzano and A. Savo, \emph{Isoperimetric inequalities for the magnetic {N}eumann and {S}teklov problems with {A}haronov-{B}ohm magnetic potential}, J. Geom. Anal. \textbf{32} (2022), no. 11, paper No. 285, 38.
\bibitem{DKSU:07} D. Dos Santos Ferreira, E.C. Kenig, J. Sj\"ostrand and G. Uhlmann, \emph{Determining a magnetic Schr\"odinger operator from partial Cauchy data}, Comm. Math. Phys. \textbf{271} (2007), no. 2, 467-488.
\bibitem{EGHP:23} M. Egidi, K. Gittins, G. Habib and N. Peyerimhoff, \emph{Eigenvalue estimates for the magnetic Hodge Laplacian on differential forms}, J. Spect. Theory \textbf{13} (2023), 1297--1343.
\bibitem{ELMP:16} M. Egidi, S. Liu, F. M\"unch and N. Peyerimhoff, \emph{Ricci curvature and eigenvalue estimates for the magnetic Laplacian on manifolds}, Comm. Anal. Geom. \textbf{29} (2021), 1127--1156.
\bibitem{ES:97} J. F. Escobar, \emph{The geometry of the first non-zero Stekloff eigenvalue}, J. Funct. Anal. \textbf{150} (1997), no. 1, 544--556.
\bibitem{EO:22} A.F.M. ter Elst and E.M. Ouhabaz, \emph{The diamagnetic inequality for the Dirichlet-to-Neumann operator}, Bull. Lond. Math. Soc. \textbf{54} (2022), no. 5, 1978–1997.
\bibitem{GKLP:22} A. Girouard, M. Karpukhin, M. Levitin and I. Polterovich, \emph{The Dirichlet-to-Neumann map, the boundary Laplacian, and H\"ormander’s rediscovered manuscript}, J. Spectr. Theory. \textbf{12} (2022), no. 2, 195--225.
\bibitem{H:18} B. Haberman, \emph{Unique determination of a magnetic Schrödinger operator with unbounded magnetic potential from boundary data}, Int. Math. Res. Not. IMRN (2018), no. 4, 1080--1128.
\bibitem{HK:18} G. Habib and A. Kachmar, \emph{Eigenvalue bounds of the Robin Laplacian with magnetic field}, Archiv. der Math. \textbf{110} (2018), 501--513.
\bibitem{Ha:59} F. Harary, \emph{On the measurement of structural balance}, Behavioral Sci. \textbf{4} (1959), 316--323.
\bibitem{He:88a} B. Helffer, \emph{Effet d'Aharonov-Bohm sur un \'etat born\'e{} de l'\'equation de Schr\"odinger}, Comm. Math. Phys. \textbf{119} (1988), no. 2, 315--329.
\bibitem{He:88b} B. Helffer, \emph{Semi-classical analysis for the Schr\"odinger operator and applications}, Lecture Notes in Mathematics, 1336, Springer, Berlin, 1988.
\bibitem{HKN:25} B. Helffer, A. Kachmar and F. Nicoleau, \emph{Asymptotics for the magnetic Dirichlet-to-Neumann eigenvalues in general domains}, $\mathtt{arXiv:2501.00947}$ (January 2025).
\bibitem{HL:24} B. Helffer and C. L\'ena, \emph{Eigenvalues of the Neumann magnetic Laplacian in the unit disk}, $\mathtt{arXiv:2411.11721}$ (November 2024).
\bibitem{HN:24a} B. Helffer and F. Nicoleau, \emph{Trace formulas for the magnetic Laplacian and Dirichlet to Neumann operator -- Explicit expansions --}, $\mathtt{arXiv:2407.08671}$ (July 2024).
\bibitem{HN:24b} B. Helffer and F. Nicoleau, \emph{On the magnetic Dirichlet to Neumann operator on the disk -- strong diamagnetism and strong magnetic field limit--}, $\mathtt{arXiv:2411.15522}$ (December 2024).
\bibitem{Hi:74} N. Hitchin, \emph{Harmonic spinors}, Advances in Math. \textbf{14} (1974), 1--55.
\bibitem{Ja:15} P. Jammes, \emph{Une in\'egalit\'e de Cheeger pour le spectre de Steklov}, Ann. inst. Fourier \textbf{65}, (2015), 1381--1385.
\bibitem{K:19} M.A. Karpukhin, \emph{The Steklov problem on differential forms}, Canad. J. Math. \textbf{71} (2019), no. 2, 417--435.
\bibitem{Ka} A. Kasue, \emph{Ricci curvature, geodesics and some geometric properties of Riemannian manifolds with boundary}, J. Math. Soc. Japan {\bf 35} (1983), no. 1, 117--131.
\bibitem{LLPP:15} C. Lange, S. Liu, N. Peyerimhoff and O. Post, \emph{Frustration index and Cheeger inequalities for discrete and continuous magnetic Laplacians}, Calc. Var.  \textbf{54}, (2015), 4165--4196.
\bibitem{LT:23} G. Liu and X. Tan, \emph{Spectral invariants of the magnetic Dirichlet-to-Neumann map on Riemannian manifolds}, J. Math. Phys. \textbf{64} (2023), no. 4, Paper No. 041501, 1-47.
\bibitem{NSU:95} G. Nakamura, Z.Q. Sun and G. Uhlmann, \emph{Global identifiability for an inverse problem for the Schr\"odinger equation in a magnetic field}, Math. Ann. \textbf{303} (1995), no. 3, 377-388.
\bibitem{Pe:93} N. Peyerimhoff, \emph{Ein Indexsatz f\"ur Cheegersingularit\"atten in Himblick auf algebraisce Fl\"achen}, PhD Thesis, Universit\"at Augsburg, Augsburg 1993.
\bibitem{PS:23} L. Provenzano and A. Savo, \emph{Geometry of the magnetic Steklov problem on Riemannian annuli}, $\mathtt{arXiv:2310.08203}$ (October 2023), to appear in Commun. Contemp. Math.
\bibitem{PS:19} L. Provenzano and J. Stubbe, \emph{Weyl-type bounds for Steklov eigenvalues}, J. Spectr. Theory \textbf{9} (2019), no. 1, 349--377.
\bibitem{RS:15} S. Raulot and A. Savo, \emph{Sharp bounds for the first eigenvalue of a fourth order Steklov problem}, J. Geom. Anal. \textbf{25}, (2015), 1602--1619. 
\bibitem{AS:14} A. Savas-Halilaj and K. Smoczyk, \emph{Bernstein theorems for length and area decreasing minimal maps}, Calc. Var. Partial Differential Equations \textbf{14}  (2014), 549--577.
\bibitem{Sch:95} G. Schwarz, \emph{Hodge decomposition-A method for solving boundary value problems}, Lecture notes
in Mathematics, Springer, 1995.
\bibitem{Shi:87} I. Shigekawa, \emph{Eigenvalue problems for the {S}chr\"{o}dinger operator with the magnetic field on a compact {R}iemannian manifold}, J. Funct. Anal. \textbf{75} (1987), no. 1, 92--127.
\bibitem{Xi:97} C. Xia, \emph{Rigidity of compact manifolds with boundary and nonnegative Ricci curvature}, Proc. Amer. Math. Soc. \textbf{125} (1997), 1801--1806.
\bibitem{Za:82} T. Zaslavsky, \emph{Signed graphs}, Discrete Appl. Math. \textbf{4} (1982), no. 1, 47--74.
\bibitem{ZR:10} T. Zaslavsky and C. Rao, \emph{Balance and clustering in signed graphs}, Slides from lectures at the CR Rao Advanced Institute of Mathematics, Statistics and Computer Science, Univ. of Hyderabad, India 26 (2010).
\end{thebibliography}
\end{document}